\documentclass[11pt]{article}

\usepackage{charter}
\usepackage[charter]{mathdesign}
\usepackage{eucal}
\usepackage{microtype}

\usepackage{subcaption}

\usepackage[margin=1in,letterpaper]{geometry}

\usepackage{enumitem}

\usepackage{natbib}

\usepackage{amsmath}
\usepackage{amsthm}
\newtheorem{thm}{Theorem}
\newtheorem{lem}{Lemma}
\newtheorem{cor}{Corollary}

\newtheorem{obs}{Observation}

\theoremstyle{definition}
\newtheorem{definition}{Definition}

\usepackage[ruled,vlined,linesnumbered]{algorithm2e}
\DontPrintSemicolon

\usepackage{ifthen,tikz}
\usetikzlibrary{positioning}
\usetikzlibrary{shapes.geometric}
\usetikzlibrary{calc}

\newcommand{\tree}[2][]{T\ifthenelse{\equal{#1}{}}{}{_{#1}}^{(#2)}}

\usepackage[colorinlistoftodos]{todonotes}

\definecolor{lightblue}{RGB}{180,180,255}

\newcommand{\leoo}[1]{#1}

\definecolor{lightgreen}{RGB}{180,255,180}

\newcommand{\remie}[1]{#1}

\begin{document}

\title{A Unifying Characterization of Tree-based Networks and Orchard Networks using Cherry Covers}
\author{Leo van Iersel\footnote{Delft Institute of Applied Mathematics, Delft
    University of Technology, Van Mourik Broekmanweg 6, 2628 XE, Delft, The
    Netherlands, \{L.J.J.vanIersel, R.Janssen-2,
      Y.Murakami\}@tudelft.nl. Research funded in part by the Netherlands
    Organization for Scientific Research (NWO) Vidi grant
    639.072.602, and partly by the 4TU Applied Mathematics Institute.} \and
  Remie Janssen\footnotemark[1] \and
  Mark Jones\footnote{Centrum Wiskunde \& Informatica (CWI), 1098 XG Amsterdam, The Netherlands, 
markelliotlloyd@gmail.com.
Research funded by by the Netherlands Organisation for Scientific Research (NWO) through Gravitation Programme Networks 024.002.003} \and
  Yukihiro Murakami\footnotemark[1] \and
  Norbert Zeh\footnote{Faculty of Computer Science, Dalhousie University, 6050
    University Ave, Halifax, NS B3H 1W5, Canada, nzeh@cs.dal.ca.  Research
    funded in part by the Natural Sciences and Engineering Research Council of
    Canada.}} 
\maketitle

\begin{abstract}
Phylogenetic networks are used to study evolutionary relationships between species in biology.
Such networks are often categorized into classes by their topological features, which stem from both biological and computational motivations.
We study two network classes in this paper: tree-based networks and orchard networks.
Tree-based networks are those that can be obtained by inserting edges between the edges of an underlying tree. \leoo{Orchard} 
networks are a recently introduced generalization of the class of tree-child networks.
Structural characterizations have already been discovered for tree-based networks; this is not the case for orchard networks.
In this paper, we introduce cherry covers---a unifying characterization of both network classes---in which we decompose the edges of the networks into so-called cherry shapes and reticulated cherry shapes.
We show that cherry covers can be used to characterize 
the class of tree-based networks \leoo{as well as} the class of orchard networks. \leoo{Moreover, we also generalize these results to non-binary networks.}
\end{abstract}

\section{Introduction}

Phylogenetic trees and networks are used to represent the evolutionary history of species in biology and languages in linguistics.
Given a set of present-day species, a tree can be used to depict how lineages have diverged from their most recent common ancestor.
Networks are a generalization of trees, and a network can also depict how lineages may have converged in the past, as a result of reticulate events such as hybridization and horizontal gene transfer.

Phylogenetic networks have been categorized into many topological classes for both biological and computational incentives (for an overview of a few binary network classes, see, for example,~\cite{huson2010phylogenetic}).
One of the largest of these network classes is the class of tree-based networks.
Hatched from an ongoing debate on whether evolutionary histories should \leoo{or should not} be viewed as tree-like with reticulate events sprinkled in (e.g., \leoo{in the context of}
horizontal gene transfer within prokaryotes~\cite{martin2011early}), tree-based networks were introduced as those that can be obtained from trees by inserting reticulate arcs between the arcs of the tree~\cite{francis2015phylogenetic}.
In their seminal paper, Francis and Steel explored the mathematical properties of these tree-based networks and provided a linear time algorithm to check whether a binary network was tree-based.
Following this, structural characterizations for binary tree-based networks were introduced in the form of forbidden substructures~\cite{zhang2016tree}, matchings~\cite{jetten2018nonbinary}, and using antichains and path-partitions~\cite{francis2018new}.
Jetten and van Iersel further extended the matching characterization result to non-binary networks, and showed that it is possible to decide if a non-binary network is tree-based in polynomial time~\cite{jetten2018nonbinary}.

We briefly comment on the difference between binary and non-binary networks.
Networks are often presented so that at each speciation event, two lineages diverge from one lineage, and at each reticulate event, two lineages converge into one lineage---this is what we would call a binary network.
In practice, many networks do not adhere to such restrictions.
For example, ambiguities in the order of how some evolutionary events have unfolded (soft polytomy) or multiple speciation events that occur almost simultaneously from a single species (hard polytomy) can easily break this ideal structure.
Such problems give rise to vertices that represent one lineage diverging into three or more lineages.
The same stands for reticulate events.
In this paper we consider non-binary networks (i.e., not necessarily binary), and therefore our results will naturally hold for binary networks.

Within the tree-based network class lies the recently introduced class of orchard networks.
These networks generalize the prominent class of tree-child networks. \leoo{It was shown that orchard networks are} uniquely reconstructible from their ancestral profiles~\cite{erdHos2019class} and that it \leoo{can} be determined whether two \remie {binary} orchard networks \leoo{are} isomorphic in polynomial time~\cite{janssen2018solving}.
Orchard networks contain either a cherry (two leaves with a common parent) or a reticulated cherry (two leaves with distinct parents, for which one parent is the parent of another, and the lower parent is a reticulation), such that reducing a cherry or a reticulated cherry yields an orchard network of smaller size.
With this reduction, one can obtain a sequence of ordered pairs 
-- which corresponds to reducing either a cherry or a reticulated cherry that involves the two leaves in the pair -- that iteratively reduces \leoo{the orchard network to} a single leaf.
Janssen and Murakami, and Erd\H{o}s et al. have independently shown that such a reduction can be done in any order, and therefore that it can be decided in linear time whether a network is orchard~\cite{erdHos2019class,janssen2018solving}.
While these sequences do characterize orchard networks, the recursive nature of this characterization may make it difficult to use.

In this paper, we present a unified structural \leoo{(non-recursive)} characterization for both non-binary tree-based networks and non-binary orchard networks.
We first decompose networks into so-called \emph{cherry shapes} and \emph{reticulated cherry shapes}.
If each edge of the network belongs to at least one of these two structures, then we say that the network has a \emph{cherry cover}.
This turns out to be a necessary and a sufficient condition for the network to be tree-based (Theorem~\ref{thm:CherryDecompositiontree-based}).
In addition we consider an ordering on the cherry  and reticulated cherry shapes of a network.
We prove that a network is orchard precisely if it has an acyclic cherry cover (Theorem~\ref{thm:CherryDecompositionorchard}).

\section{Preliminaries}

A \emph{(phylogenetic non-binary) network} is a directed acyclic graph with a unique vertex of indegree-0 and outdegree-1 (the \emph{root}), vertices of indegree-1 and outdegree-0 that are bijectively labelled by~$X$ (the \emph{leaves}), and all other vertices have either indegree-1 (\emph{tree vertices}) or outdegree-1 (\emph{reticulations}) but not both.
A \emph{(phylogenetic) tree} is a network with no reticulations.

Given an edge~$uv$ in a network, we say that~$u$ is a \emph{parent} of~$v$ and that~$v$ is a \emph{child} of~$u$.
We say that~$u$ and~$v$ are the \emph{tail} and \emph{head} of the edge, respectively.
An edge~$uv$ is a \emph{reticulation edge} if the vertex~$v$ is a reticulation, and it is the \emph{root edge} if~$u$ is the root.
The \emph{reticulation number} is the total number of reticulation edges minus the total number of reticulation vertices.
A vertex in a network is \emph{binary} if it has total degree at most three.
A binary tree vertex is called a \emph{bifurcation} and a non-binary tree vertex is a \emph{multifurcation}.
A network is \emph{semi-binary} if all tree vertices are binary; it is \emph{binary} if all vertices are binary.
A \emph{semi-binary resolution} of a network $N$ is a semi-binary network that can be turned into $N$ by contracting edges. 
A \emph{binary resolution} of a network $N$ is a binary network that can be turned into $N$ by contracting edges.

\subsection{Cherry cover}

A \emph{cherry shape} is a subgraph on three distinct vertices~$x,y,p$ with edges~$px$ and~$py$.
The \emph{internal vertex} of a cherry shape is~$p$, and the \emph{endpoints} are~$x$ and~$y$.
A \emph{reticulated cherry shape} is a subgraph on four vertices~$x,y,p_x,p_y$ with edges~$p_xx, p_yp_x,p_yy$, such that~$p_x$ is a reticulation in the network.
The \emph{internal vertices} of a reticulated cherry shape are~$p_x$ and~$p_y$, and the \emph{endpoints} are~$x$ and~$y$.
The \emph{internal reticulation} and the \emph{middle edge} of a reticulated cherry shape is~$p_x$ and~$p_yp_x$, respectively.
The edge~$p_yy$ is called the \emph{free edge} of the reticulated cherry shape.
We will often refer to cherry shapes and the reticulated cherry shapes by their edges (e.g., we would denote the above cherry shape~$\{px,py\}$ and the reticulated cherry shape~$\{p_xx, p_yp_x, p_yy\}$).
We say that an edge~$uv$ is \emph{covered} by a cherry or reticulated cherry shape~$C$ if~$uv\in C$.
Given a set~$P$ of cherry and reticulated cherry shapes, we say that an edge is \emph{covered} by~$P$ if the edge is covered by at least one shape in~$P$.
We now investigate how sets of cherry shapes and reticulated cherry shapes may form a \emph{decomposition / cover} for a given binary / semi-binary / non-binary network.

\subsubsection{Binary networks}

\begin{definition}
	A \emph{cherry decomposition} of a binary network is a set~$P$ of cherry shapes and reticulated cherry shapes, such that each edge except for the root edge is covered exactly once by~$P$.
\end{definition}
An example of a binary network with its cherry decomposition is presented in Figure~\ref{fig:NSBBNetworks}.
\begin{lem}\label{lem:CherryCoverCount}
	Let~$N$ be a binary network on~$n$ leaves and reticulation number~$r$, and let~$P$ be a cherry decomposition of~$N$.
	Then~$P$ contains exactly~$n-1$ cherry shapes and~$r$ reticulated cherry shapes.
\end{lem}
\begin{proof}
	The network~$N$ contains~$n-1+r$ tree vertices,~$r$ reticulation vertices, and~$n+1$ vertices of degree~$1$ (including the root).
	By the Handshaking Lemma, we get that the total number of edges in~$N$ is~$2n+3r-1$.
	Then the total number of edges of~$N$ excluding the root edge is~$2(n-1) + 3r$.
	Observe that every outgoing edge of a reticulation vertex must be covered by a reticulated cherry shape.
	Since there are~$r$ such edges and because a reticulated cherry shape is composed of~$3$ edges, we have that~$3r$ of the edges of~$N$ are covered by reticulated cherry shapes, and that the rest of the edges of~$N$ must be covered by cherry shapes.
	As each cherry shape is composed of~$2$ edges, this implies that there must be~$n-1$ cherry shapes in~$P$.
	Therefore~$P$ contains exactly~$n-1$ cherry shapes and~$r$ reticulated cherry shapes.
\end{proof}

\subsubsection{Semi-binary networks}

We extend the notion of a cherry decomposition to semi-binary networks by introducing the following ``bulged version'' of a network.
\begin{definition}
    Let $N$ be a network. Then the \emph{bulged version} of $N$, $B(N)$, is the multigraph obtained from $N$ by replacing the outgoing edge of each reticulation vertex with indegree-$k$ by $k-1$ parallel edges. If $M$ is the bulged version of a network $M'$, then we define $B^{-1}(M)=M'$.
\end{definition}
If $N$ is binary, we have $N=B(N)$.
For any network~$N$, we have~$B^{-1}(B(N)) = N$.
Observe that in general, bulged versions of networks are not always networks.
In bulged versions of networks, we call vertices of indegree at least~$2$ and outdegree at least~$1$ a \emph{reticulation} (note that reticulations in this case may be vertices of outdegree more than~$1$).
\emph{Roots} are vertices of indegree-$0$, \emph{tree vertices} are of indegree-$1$, and \emph{leaves} are of outdegree-$0$.

\begin{definition}
	A \emph{cherry decomposition} of the bulged version of a semi-binary network~$N$ is a set~$P$ of cherry shapes and reticulated cherry shapes, such that each edge of~$B(N)$, except for the root edge, is covered exactly once by~$P$.
\end{definition}
Observe that a reticulation vertex in the bulged version of the network is always mapped to an internal reticulation of a reticulated cherry shape in the cherry decomposition.
This brings us to the following lemma.
\begin{lem}
	Let~$N$ be a semi-binary network on~$n$ leaves and reticulation number~$r$, and let~$P$ be a cherry decomposition of~$N$.
	Then~$P$ contains exactly~$n-1$ cherry shapes and~$r$ reticulated cherry shapes.
\end{lem}
\begin{proof}
	The proof follows an analogous argument used in the proof of Lemma~\ref{lem:CherryCoverCount}.
\end{proof}

\subsubsection{Non-binary networks}

For non-binary networks, we generalize the concept of cherry decompositions by allowing certain edges to be covered multiple times.

\begin{definition}\label{def:CherryCover}
    A \emph{cherry cover} of the bulged version of a non-binary network~$N$ is a set~$P$ of cherry shapes and reticulated cherry shapes with the following properties on~$B(N)$:
    \begin{itemize}
	\item each edge except for the root edge is covered by at least one shape in~$P$,
	\item each outgoing edge of a reticulation vertex is covered exactly once,
	\item each edge covered by the middle edge of a reticulated cherry shape is covered exactly once.
    \end{itemize}
\end{definition}

Note that cherry covers may contain cherry shapes that cover the same edge of the bulged version of the network, as long as the above properties are respected (see Figure~\ref{fig:CoverofMultifurcations}).
Note also that there may exist many distinct cherry covers for the same network.

\begin{figure}[]
	\begin{subfigure}[b]{.33\textwidth}
	\centering
	\resizebox{\linewidth}{!}{
	\begin{tikzpicture}[ every node/.style={draw, circle, fill, inner sep=0pt},
	square/.style={regular polygon, regular polygon sides=4}]
       	\tikzset{edge/.style={very thick}}
	       	\begin{scope}[xshift=0cm,yshift=0cm]
       		\node[] (root) at (0,5)	{a};
	       	
       		\node[below = 0.5 of root] 				(t1)	{a};
			\node[below left = 0.75 and 0.5 of t1]		(t2)	{a};
			\node[below left = 0.75 and 0.5 of t2]		(t3)	{a};
			\node[right = 2.5 of t3]				(t5)	{a};
			
			\node[square, below = 1.75 of t2]	(r1)	{a};
			\node[square, right = 1 of r1]		(r2)	{a};

			\node[below = 1 of r1]		(l1)	{a};
			\node[below = 1 of r2]		(l2)	{a};
			
			\draw[edge]			(root) -- (t1) node[draw = none, fill = none, left = 5mm, midway]{\large $N$};
			\draw[edge]			(t1)--(t2);
			\draw[edge]			(t1)--(t5);
			\draw[edge]			(t2)--(t3) node[draw = none, fill = none, above left = 1mm, midway]{\large $e$};
			\draw[edge]			(t2)--(r1);
			\draw[edge]			(t2)--(r2);
			
			\draw[edge]			(t3)--(r1);
			\draw[edge]			(t5)--(r1);
			
			\draw[edge]			(t3)--(r2);
			\draw[edge]			(t5)--(r2);
			
			\draw[edge]			(r1)--(l1);
			\draw[edge]			(r2)--(l2);
			
            		\node[draw = none, fill = none, below=1mm of l1]  	(a) {\large $a$};
            		\node[draw = none, fill = none, below=1mm of l2]  	(b) {\large $b$};
			
			\end{scope}
			\begin{scope}[xshift=3.5cm,yshift=0cm]
       		\node[] (root) at (0,5)	{a};
	       	
       		\node[below = 0.5 of root] 				(t1)	{a};
			\node[below left = 0.75 and 0.5 of t1]		(t2)	{a};
			\node[below left = 0.75 and 0.5 of t2]		(t3)	{a};
			\node[right = 2.5 of t3]				(t5)	{a};
			
			\node[square, below = 1.75 of t2]	(r1)	{a};
			\node[square, right = 1 of r1]		(r2)	{a};

			\node[below = 1 of r1]		(l1)	{a};
			\node[below = 1 of r2]		(l2)	{a};
			
			\draw[edge]			(root) -- (t1) node[draw = none, fill = none, left = 5mm, midway]{\large $B(N)$};
			\draw[edge]			(t1)--(t2);
			\draw[edge]			(t1)--(t5);
			\draw[edge, dashed]			(t2.west)--(t3.west);
			\draw[edge, dotted]			(t2.east)--(t3.east);
			\draw[edge, densely dashdotted]			(t5)--(r1);
			\draw[edge, dotted]			(t2)--(r2);
			
			\draw[edge]			(t3)--(r1);
			\draw[edge, dashed]			(t2)--(r1);
			
			\draw[edge]			(t3)--(r2);
			\draw[edge, densely dashdotted]			(t5)--(r2);
			
			\draw[edge, bend left, dashed]			(r1)edge(l1);
			\draw[edge, bend right, densely dashdotted]				(r1)edge(l1);
			\draw[edge, bend left]			(r2)edge(l2);
			\draw[edge, bend right, dotted]		(r2)edge(l2);
			
            		\node[draw = none, fill = none, below=1mm of l1]  	(a) {\large $a$};
            		\node[draw = none, fill = none, below=1mm of l2]  	(b) {\large $b$};
			
			\end{scope}
	 	\end{tikzpicture}}
	 	\caption{}
 	\end{subfigure}
 	\begin{subfigure}[b]{.33\textwidth}
 	\centering
 	\resizebox{\linewidth}{!}{
	\begin{tikzpicture}[every node/.style={draw, circle, fill, inner sep=0pt},
	square/.style={regular polygon, regular polygon sides=4}]
       	\tikzset{edge/.style={very thick}}
	       	\begin{scope}[xshift=0cm,yshift=0cm]
       		\node[] (root) at (0,5)	{a};
	       	
       		\node[below = 0.5 of root] 				(t1)	{a};
			\node[below left = 0.75 and 0.5 of t1]		(t2)	{a};
			\node[below left = 0.75 and 0.5 of t2]		(t3)	{a};
			\node (t4) at ($(t1)!0.5!(t2)$) {a};
			\node[right = 2.5 of t3]				(t5)	{a};
			
			\node[square, below = 1.75 of t2]	(r1)	{a};
			\node[square, right = 1 of r1]		(r2)	{a};

			\node[below = 1 of r1]		(l1)	{a};
			\node[below = 1 of r2]		(l2)	{a};
			
			\draw[edge]			(root) -- (t1) node[draw = none, fill = none, left = 5mm, midway]{\large $N^s$};
			\draw[edge]			(t1)--(t2);
			\draw[edge]			(t1)--(t5);
			\draw[edge]			(t2)--(t3);
			\draw[edge]			(t2)--(r1);
			\draw[edge]			(t4)--(r2);
			
			\draw[edge]			(t3)--(r1);
			\draw[edge]			(t5)--(r1);
			
			\draw[edge]			(t3)--(r2);
			\draw[edge]			(t5)--(r2);
			
			\draw[edge]			(r1)--(l1);
			\draw[edge]			(r2)--(l2);
			
            		\node[draw = none, fill = none, below=1mm of l1]  	(a) {\large $a$};
            		\node[draw = none, fill = none, below=1mm of l2]  	(b) {\large $b$};
			\end{scope}
			\begin{scope}[xshift=3.5cm,yshift=0cm]
			\node[] (root) at (0,5)	{a};
	       	
       		\node[below = 0.5 of root] 				(t1)	{a};
			\node[below left = 0.75 and 0.5 of t1]		(t2)	{a};
			\node[below left = 0.75 and 0.5 of t2]		(t3)	{a};
			\node (t4) at ($(t1)!0.5!(t2)$) {a};
			\node[right = 2.5 of t3]				(t5)	{a};
			
			\node[square, below = 1.75 of t2]	(r1)	{a};
			\node[square, right = 1 of r1]		(r2)	{a};

			\node[below = 1 of r1]		(l1)	{a};
			\node[below = 1 of r2]		(l2)	{a};
			
			\draw[edge]			(root) -- (t1) node[draw = none, fill = none, left = 5mm, midway]{\large $B(N^s)$};
			\draw[edge]			(t1)--(t4);
			\draw[edge]			(t1)--(t5);
			\draw[edge, dashed]			(t2)--(t3);
			\draw[edge, dashed]			(t2)--(r1);
			\draw[edge, dotted]			(t4)--(r2);
			\draw[edge, dotted]			(t4)--(t2);
			
			\draw[edge]			(t3)--(r1);
			\draw[edge, densely dashdotted]			(t5)--(r1);
			
			\draw[edge]			(t3)--(r2);
			\draw[edge, densely dashdotted]			(t5)--(r2);
			
			\draw[edge, bend left, dashed]			(r1)edge(l1);
			\draw[edge, bend right, densely dashdotted]				(r1)edge(l1);
			\draw[edge, bend left]			(r2)edge(l2);
			\draw[edge, bend right, dotted]		(r2)edge(l2);
			
            		\node[draw = none, fill = none, below=1mm of l1]  	(a) {\large $a$};
            		\node[draw = none, fill = none, below=1mm of l2]  	(b) {\large $b$};
			\end{scope}
	 	\end{tikzpicture}}
	 	\caption{}
 	\end{subfigure}
 	\begin{subfigure}[b]{.33\textwidth}
	\centering
	\resizebox{\linewidth}{!}{
	\begin{tikzpicture}[every node/.style={draw, circle, fill, inner sep=0pt},
	square/.style={regular polygon, regular polygon sides=4}]
       	\tikzset{edge/.style={very thick}}
	       	\begin{scope}[xshift = 0cm, yshift=0cm]
       		\node[] (root) at (0,5)	{a};
	       	
       		\node[below = 0.5 of root] 				(t1)	{a};
			\node[below left = 0.75 and 0.5 of t1]		(t2)	{a};
			\node[below left = 0.75 and 0.5 of t2]		(t3)	{a};
			\node (t4) at ($(t1)!0.5!(t2)$) {a};
			\node[right = 2.5 of t3]				(t5)	{a};
			
			\node[square, below = 1.75 of t2]	(r1)	{a};
			\node[square, right = 1 of r1]		(r2)	{a};
			\node[square] (r3) at ($(t3)!0.5!(r1)$) {a};
			\node[square] (r4) at ($(t5)!0.5!(r2)$) {a};

			\node[below = 1 of r1]		(l1)	{a};
			\node[below = 1 of r2]		(l2)	{a};
			
			\draw[edge]			(root) -- (t1) node[draw = none, fill = none, left = 5mm, midway]{\large $N^b$};
			\draw[edge]			(t1)--(t4);
			\draw[edge]			(t4)--(t2);
			\draw[edge]			(t1)--(t5);
			\draw[edge]			(t2)--(t3);
			\draw[edge]			(t2)--(r3);

			\draw[edge]			(t3)--(r1);
			\draw[edge]			(t5)--(r1);
			
			\draw[edge]			(t3)--(r2);
			\draw[edge]			(t5)--(r2);
			
			\draw[edge]			(t4)--(r4);
			
			\draw[edge]			(r1)--(l1);
			\draw[edge]			(r2)--(l2);
			
            		\node[draw = none, fill = none, below=1mm of l1]  	(a) {\large $a$};
            		\node[draw = none, fill = none, below=1mm of l2]  	(b) {\large $b$};
			\end{scope}
			\begin{scope}[xshift = 3.5cm, yshift = 0cm]
       		\node[] (root) at (0,5)	{a};
	       	
       		\node[below = 0.5 of root] 				(t1)	{a};
			\node[below left = 0.75 and 0.5 of t1]		(t2)	{a};
			\node[below left = 0.75 and 0.5 of t2]		(t3)	{a};
			\node (t4) at ($(t1)!0.5!(t2)$) {a};
			\node[right = 2.5 of t3]				(t5)	{a};
			
			\node[square, below = 1.75 of t2]	(r1)	{a};
			\node[square, right = 1 of r1]		(r2)	{a};
			\node[square] (r3) at ($(t3)!0.5!(r1)$) {a};
			\node[square] (r4) at ($(t5)!0.5!(r2)$) {a};

			\node[below = 1 of r1]		(l1)	{a};
			\node[below = 1 of r2]		(l2)	{a};
			
			\draw[edge]			(root) -- (t1) node[draw = none, fill = none, left = 5mm, midway]{\large $N^b$};
			\draw[edge]			(t1)--(t4);
			\draw[edge, dotted]			(t4)--(t2);
			\draw[edge]			(t1)--(t5);
			\draw[edge, dashed]			(t2)--(t3);
			\draw[edge, dashed]			(t2)--(r3);

			\draw[edge]			(t3)--(r3);
			\draw[edge, dashed]			(r3)--(r1);
			\draw[edge, densely dashdotted]			(t5)--(r1);
			
			\draw[edge]			(t3)--(r2);
			\draw[edge, densely dashdotted]			(t5)--(r4);
			\draw[edge, dotted]			(r4)--(r2);
			
			\draw[edge, dotted]			(t4)--(r4);
			
			\draw[edge, densely dashdotted]			(r1)--(l1);
			\draw[edge]			(r2)--(l2);
			
            		\node[draw = none, fill = none, below=1mm of l1]  	(a) {\large $a$};
            		\node[draw = none, fill = none, below=1mm of l2]  	(b) {\large $b$};
			\end{scope}
	 	\end{tikzpicture}}
	 	\caption{}
 	\end{subfigure}
	\caption{All arcs in networks are directed from the root towards the leaves (down the page).
   	Reticulations are indicated by square vertices.
	(a) A non-binary network~$N$ and a bulged version~$B(N)$ of~$N$.
	Observe that the leaves~$a$ and~$b$ are incident to parallel edges in~$B(N)$, as they are both adjacent to reticulation vertices of indegree-$3$.
	A cherry cover of~$N$ is visualized on the edges of~$B(N)$ by different edge types.
	The edge~$e$ in~$N$ is duplicated in~$B(N)$ to depict what happens when an edge is covered twice by a cherry cover.
	(b) A semi-binary resolution~$N^s$ of~$N$, obtained by resolving the multifurcation in~$N$.
	The bulged version of~$N^s$ is shown on the right, together with the cherry decomposition of~$N^s$.
	(c) A binary resolution~$N^b$ of~$N$. A cherry decomposition of~$N^b$ is displayed on the right network.}
	\label{fig:NSBBNetworks}
\end{figure}
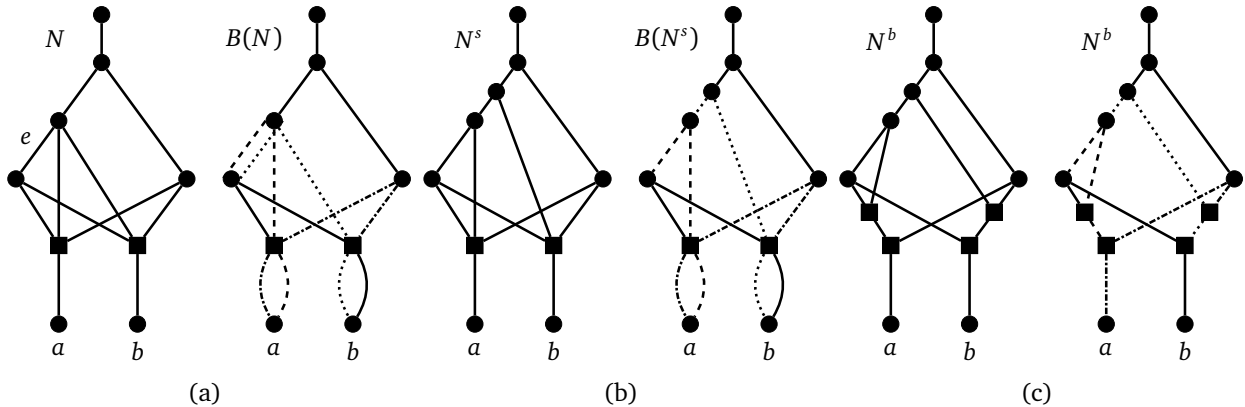

\begin{lem}\label{lem:DoubleCoveredEdge}
	Let~$P$ be a cherry cover for a non-binary network~$N$, and let~$uv$ be an edge of~$B(N)$ that is covered by at least two shapes in~$P$.
	Then~$u$ must be a multifurcation.
\end{lem}
\begin{proof}
	Observe first that~$u$ cannot be the root since the root edge is not covered by~$P$, and it also cannot be a vertex of outdegree-$0$.
	Furthermore,~$u$ cannot be a reticulation vertex by definition (second condition of Definition~\ref{def:CherryCover}).
	Therefore~$u$ must be a tree vertex.
	Suppose that~$u$ is a bifurcation, and let~$uw$ be an edge of~$B(N)$ that is not~$uv$.
	Then the edges~$uv$ and~$uw$ must be contained in the same shape~$A$ in~$P$.
	But then no shape of~$P$ other than~$A$ can contain the edge~$uv$ without violating the third condition of the cherry cover definition.
	Thus~$u$ cannot be a bifurcation.
	It follows that~$u$ must be a multifurcation.
\end{proof}

It follows that cherry covers are indeed a generalization of cherry decompositions, since a cherry cover of a binary or a semi-binary network covers each edge of the (bulged version of the) network exactly once.
Observe that the converse of Lemma~\ref{lem:DoubleCoveredEdge} is not necessarily true.
That is, given a cherry cover of a network, it is not always the case that an outgoing edge of a multifurcation is covered more than once (see Figure~\ref{fig:CoverofMultifurcations}).

%
%

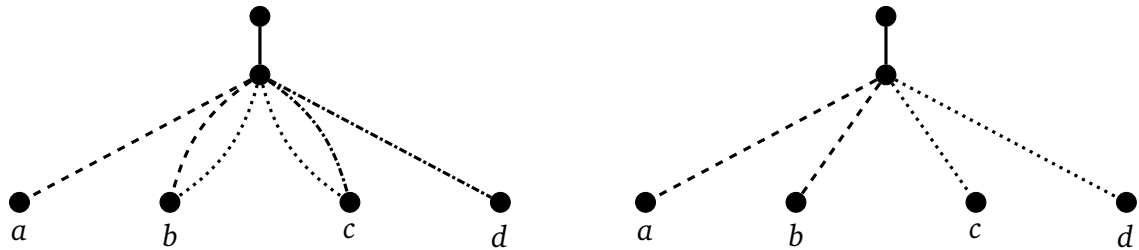
\begin{figure}
	\begin{subfigure}[b]{.5\textwidth}
		\centering
		\begin{tikzpicture}[every node/.style={draw, circle, fill, inner sep=0pt},
		square/.style={regular polygon, regular polygon sides=4}]
	       	\tikzset{edge/.style={very thick}}
	       	
	       	\node[] (root) at (0,5)	{a};
	       	
	       	\node[below = 0.5 of root] 			(t1)	{a};
			\node[below left = 1.5 and 3 of t1] 	(l1) {a};
			\node[below left = 1.5 and 1 of t1] 	(l2) {a};
			\node[below right = 1.5 and 1 of t1] 	(l3) {a};
			\node[below right= 1.5 and 3 of t1] 	(l4) {a};
			
			\draw[edge]			(root) -- (t1);

			\draw[edge, dashed]							(t1)--(l1);
			\draw[edge, bend right=20, dashed]			(t1)edge(l2);
			\draw[edge, bend left=20, dotted]				(t1)edge(l2);
			\draw[edge, bend right=20, dotted]				(t1)edge(l3);
			\draw[edge, bend left=20, densely dashdotted]	(t1)edge(l3);
			\draw[edge, densely dashdotted]				(t1)--(l4);
			
                	\node[draw = none, fill = none, below=1mm of l1]  	(a) {\large $a$};
                	\node[draw = none, fill = none, below=1mm of l2]  	(b) {\large $b$};
                	\node[draw = none, fill = none, below=1mm of l3]	(c) {\large $c$};
                	\node[draw = none, fill = none, below=1mm of l4]	(c) {\large $d$};
			
	 	\end{tikzpicture}
 	\end{subfigure}
	\begin{subfigure}[b]{.5\textwidth}
		\centering
		\begin{tikzpicture}[every node/.style={draw, circle, fill, inner sep=0pt},
		square/.style={regular polygon, regular polygon sides=4}]
	       	\tikzset{edge/.style={very thick}}
	       	
	       	\node[] (root) at (0,5)	{a};
	       	
	       	\node[below = 0.5 of root] 			(t1)	{a};
			\node[below left = 1.5 and 3 of t1] 	(l1) {a};
			\node[below left = 1.5 and 1 of t1] 	(l2) {a};
			\node[below right = 1.5 and 1 of t1] 	(l3) {a};
			\node[below right= 1.5 and 3 of t1] 	(l4) {a};
			
			\draw[edge]			(root) -- (t1);

			\draw[edge, dashed]			(t1)--(l1);
			\draw[edge, dashed]			(t1)--(l2);
			\draw[edge, dotted]			(t1)--(l3);
			\draw[edge, dotted]			(t1)--(l4);
			
                	\node[draw = none, fill = none, below=1mm of l1]  	(a) {\large $a$};
                	\node[draw = none, fill = none, below=1mm of l2]  	(b) {\large $b$};
                	\node[draw = none, fill = none, below=1mm of l3]	(c) {\large $c$};
                	\node[draw = none, fill = none, below=1mm of l4]	(c) {\large $d$};
			
	 	\end{tikzpicture}
 	\end{subfigure}
    \caption{Cherry covers of sizes~$3$ (left) and~$2$ (right) for the same tree.
    We duplicate the edges incident to~$b$ and~$c$ to show how an edge can be covered more than once in a cherry cover.
    The cherry cover of the left tree reflects the cherry cover used in the proof of Lemma~\ref{lem:OnlyCherryCoverThenTree}.}
    \label{fig:CoverofMultifurcations}
\end{figure}

\begin{lem}\label{lem:OnlyCherryCoverThenTree}
    Let $N$ be a network on~$n$ leaves. Then $B(N) = N$ has a cherry cover using only cherry shapes if and only if $N$ is a tree.
    Furthermore, if~$N$ is a tree, then there exists a cherry cover of~$N$ that contains exactly~$n-1$ cherry shapes.
\end{lem}
\begin{proof}
	The first statement follows from the definition of a cherry cover.
	To show the second statement, we construct a cherry cover for~$N$ as follows.
	Let~$t$ be a tree vertex in~$N$ of outdegree-$d$.
	Arbitrarily enumerate the~$d$ outgoing edges of~$t$ by~$e_1,e_2,\ldots, e_d$, and define cherry shapes~$C_{t_i} = \{e_i,e_{i+1}\}$ for~$i\in[d-1] = \{1,\ldots, d-1\}$.
	These~$d-1$ cherry shapes cover all outgoing edges of~$t$.
	We repeat this for all tree vertices, and since the tail of every edge, except for the root edge, is a tree vertex, we obtain a cherry cover.
	
	Let~$T(N)$ denote the tree vertices of~$N$.
	Since the sum of the indegrees is equal to the sum of the outdegrees, we get that
	\[n + |T(N)| = \sum_{v\in N}indeg(v) = \sum_{v\in N}outdeg(v) = 1 + \sum_{t\in T(N)}outdeg(t).\]
	Rearranging this equation, we find
	\[\sum_{t\in T(N)}outdeg(t)-|T(N)| = n-1.\]
	Since each tree vertex~$t$ is responsible for crafting~$outdeg(t)-1$ cherry shapes, we have that the size of the cherry cover is exactly~$\sum_{t\in T(N)}outdeg(t) - |T(N)| = n-1$, which is what we wanted.
\end{proof}

\begin{definition}\label{def:CherryShapeOrder}
    Let $P$ be a cherry cover of some network. A shape $A\in P$ is \emph{directly above} another shape $B\in P$ if an internal vertex of $B$ is an endpoint of $A$. A shape $A\in P$ is \emph{above} a shape $B\in P$ if there is a sequence $A=A_0,\ldots,A_n=B$ such that $A_{i-1}$ is directly above $A_i$ for all $i\in[n]$. 
    The cherry cover $P$ is called \emph{acyclic} if no shape is above itself.
\end{definition}

Given a cherry cover of some network, Definition~\ref{def:CherryShapeOrder} naturally gives rise to an auxiliary graph where the cherry shapes and reticulated cherry shapes are the vertices and an edge is drawn from one shape to another if it is directly above the shape.
It can be used to determine the acyclicity of a cherry cover.
An example of such a graph can be seen in Figure~\ref{subfig:AuxGraph}.

\subsection{Reducing shapes}

Let~$N$ be a network and let~$B(N)$ be the bulged version of~$N$.
We now define the action of \emph{reducing} cherry shapes and reticulated cherry shapes in~$N$, when the two endpoints of these shapes are leaves.
Let~$A$ be a cherry shape or a reticulated cherry shape on leaf endpoints~$x,y$ in~$B(N)$.
Then~$B(N)\setminus A$ is the graph obtained by deleting one edge for every element of~$A$ from~$B(N)$, deleting all isolated vertices, and labelling all unlabelled outdegree-$0$ vertices by the label of one of their children in~$N$.
In the case that~$A$ is a cherry shape and that the parent of~$y$ is a bifurcation, then we have the option of labelling the outdegree-$0$ vertex as either~$x$ or~$y$; when we reduce cherry shapes, we generally reduce it as an ordered pair---in this case~$(x,y)$---in which case we label the outdegree-$0$ vertex as~$y$.

If~$p_y$, the parent of~$y$ in~$N$ is a multifurcation, then observe that the vertex~$p_y$ becomes a vertex of indegree-$1$ and outdegree at least~$1$ in~$B(N)\setminus A$.
In this case, we let~$(B(N)\setminus A)\cup \{p_yy\}$ denote the graph obtained by adding a vertex with the label~$y$ and adding an edge~$\{p_yy\}$.
By making sure to add this leaf and edge to~$B(N)\setminus A$ whenever the parent of~$y$ is a multifurcation, we ensure that the resulting graph is always a bulged version of some network.

\begin{definition}
Let~$N$ be a network and let~$A$ be a cherry shape or a reticulated cherry shape whose endpoints are leaves~$x,y$ in~$B(N)$.
Let~$p_y$ denote the parent of~$y$ in~$N$.
Then \emph{reducing}~$A$ from~$N$ is the action of obtaining the graph
\begin{itemize}
	\item $B(N)\setminus A$ if~$p_y$ is a bifurcation;
	\item $(B(N)\setminus A)\cup \{p_yy\}$ if~$p_y$ is a multifurcation.
\end{itemize}
\end{definition}



\subsection{Network Classes}

\paragraph{Tree-based networks} 
We use the definition of \emph{non-binary tree-based networks} from Jetten and van Iersel~\cite{jetten2018nonbinary}.
Note that in their paper, they define two variants of tree-basedness of non-binary networks: one called ``tree-based'' and the other ``strictly tree-based''.
Here, we focus on the former definition.
\begin{definition}
    A network~$N$ is \emph{tree-based} with \emph{base tree}~$T$ when~$N$ can be obtained from~$T$ via the following steps:
    \begin{enumerate}
        \item Replace some edges of~$T$ by paths, whose internal vertices are called \emph{attachment points}. Attachment points have indegree-$1$ and outdegree-$1$.
        \item Add arcs, called \emph{linking arcs}, between pairs of attachment points and from tree vertices to attachments points, so that~$N$ remains acyclic, attachment points have indegree or outdegree~$1$, and~$N$ has no parallel arcs. 
        \item Suppress every attachment point that is not incident to a linking arc.
    \end{enumerate}
\end{definition}
See Figure~\ref{fig:TBCherryCover} for an example of a tree-based network, its bulged version, and a cherry cover for the network.

Given a tree-based network~$N$, we may reverse the above actions by deleting enough reticulation edges---call this set of edges~$E_r$---and suppressing all indegree-1 outdegree-1 vertices to obtain a base tree~$T$ (note that~$E_r$ may not necessarily be unique).
Letting~$V(N)$ and~$E(N)$ denote the vertices and the edges of~$N$ respectively, we define the \emph{embedding} of~$T$ in~$N$ by the subgraph of~$N$ with vertex set~$V(N)$ and edge set~$E(N) \setminus E_r$.
Observe that suppressing all indegree-1 outdegree-1 vertices from the embedding of~$T$ in~$N$ returns the tree~$T$.

Let~$N$ be a network on~$X$.
We say that the bulged version of~$N$,~$B(N)$, is \emph{tree-based} if the leaves of some spanning tree of~$B(N)$ are labelled bijectively by~$X$.
Because a spanning tree of~$B(N)$ contains exactly one edge from each set of parallel edges, we come to the following observation.
\begin{obs}\label{obs:BulgeOfTBisTB}
    A network~$N$ is tree-based if and only if~$B(N)$ is tree-based.
\end{obs}

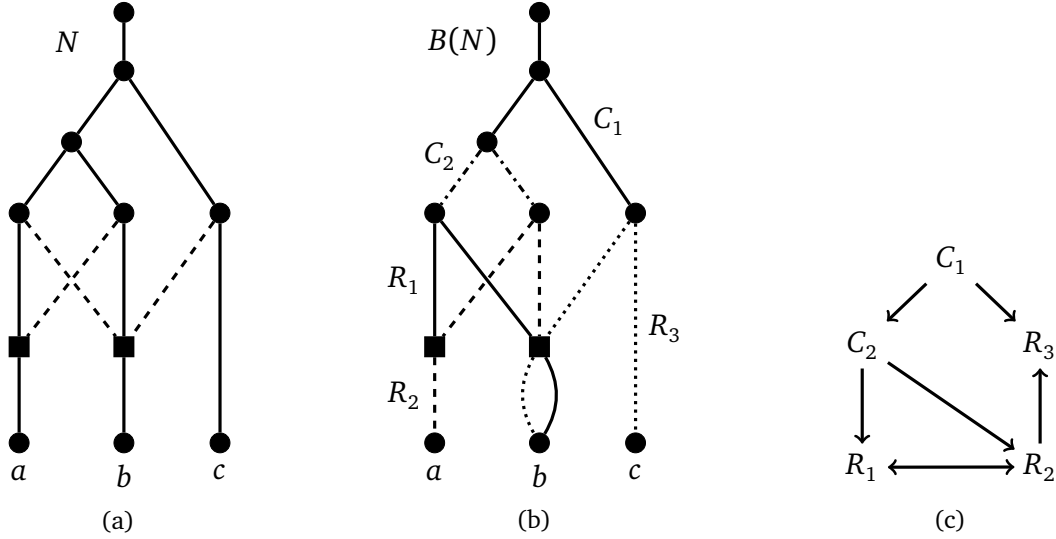
\begin{figure}
	\begin{subfigure}[b]{.33\textwidth}
		\centering
		\begin{tikzpicture}[every node/.style={draw, circle, fill, inner sep=0pt},
		square/.style={regular polygon, regular polygon sides=4}]
	       	\tikzset{edge/.style={very thick}}
	       	
	       		\node[] (root) at (0,5)	{a};
	       	
	       		\node[below = 0.5 of root] 				(t1)	{a};
			\node[below left = 0.75 and 0.5 of t1]		(t2)	{a};
			\node[below left = 0.75 and 0.5 of t2]		(t3)	{a};
			\node[below right = 0.75 and 0.5 of t2]	(t4)	{a};
			\node[right = 1 of t4]					(t5)	{a};
			
			\node[square, below = 1.5 of t3]	(r1)	{a};
			\node[square, below = 1.5 of t4]	(r2)	{a};

			\node[below = 1 of r1]		(l1)	{a};
			\node[below = 1 of r2]		(l2)	{a};
			\node[right = 1 of l2]			(l3)	{a};
			
			\draw[edge]			(root) -- (t1)  node[draw = none, fill = none, left= 5mm, midway]{\large $N$};
			\draw[edge]			(t1)--(t2);
			\draw[edge]			(t1)--(t5);
			\draw[edge]			(t2)--(t3);
			\draw[edge]			(t2)--(t4);
			
			\draw[edge]			(t3)--(r1);
			\draw[edge]			(t4)--(r2);
			
			\draw[edge, dashed]		(t3)--(r2);
			\draw[edge, dashed]		(t4)--(r1);
			\draw[edge, dashed]		(t5)--(r2);
			
			\draw[edge]			(r1)--(l1);
			\draw[edge]			(r2)--(l2);
			\draw[edge]			(t5)--(l3);
			
                		\node[draw = none, fill = none, below=1mm of l1]  	(a) {\large $a$};
                		\node[draw = none, fill = none, below=1mm of l2]  	(b) {\large $b$};
                		\node[draw = none, fill = none, below=1mm of l3]	(c) {\large $c$};
			
	 	\end{tikzpicture}
	 	\caption{}
 	\end{subfigure}
 	\begin{subfigure}[b]{.33\textwidth}
		\centering
		\begin{tikzpicture}[every node/.style={draw, circle, fill, inner sep=0pt},
		square/.style={regular polygon, regular polygon sides=4}]
	       	\tikzset{edge/.style={very thick}}
	       	
	       		\node[] (root) at (0,5)	{a};
	       	
	       		\node[below = 0.5 of root] 				(t1)	{a};
			\node[below left = 0.75 and 0.5 of t1]		(t2)	{a};
			\node[below left = 0.75 and 0.5 of t2]		(t3)	{a};
			\node[below right = 0.75 and 0.5 of t2]	(t4)	{a};
			\node[right = 1 of t4]					(t5)	{a};
			
			\node[square, below = 1.5 of t3]	(r1)	{a};
			\node[square, below = 1.5 of t4]	(r2)	{a};

			\node[below = 1 of r1]	(l1)	{a};
			\node[below = 1 of r2]	(l2)	{a};
			\node[right = 1 of l2]		(l3)	{a};
			
			\draw[edge]				(root) -- (t1) node[draw = none, fill = none, left= 5mm, midway]{\large $B(N)$};
			\draw[edge]				(t1)--(t2);
			\draw[edge]				(t1)--(t5) node[draw = none, fill = none, above right= 1mm, midway]{\large $C_1$};
			\draw[edge, dashdotted]			(t2)--(t3) node[draw = none, fill = none, above left = 1mm, midway]{\large $C_2$};
			\draw[edge, dashdotted]			(t2)--(t4);
			
			\draw[edge]				(t3)--(r1) node[draw = none, fill = none, left = 1mm, midway]{\large $R_1$};
			\draw[edge, dashed]			(t4)--(r2);
			
			\draw[edge]				(t3)--(r2);
			\draw[edge, dashed]			(t4)--(r1);
			\draw[edge, dotted]			(t5)--(r2);
			
			\draw[edge, dashed]			(r1)--(l1) node[draw = none, fill = none, left = 1mm, midway]{\large $R_2$};
			\draw[edge, bend right, dotted]	(r2) edge (l2);
			\draw[edge, bend left]		(r2) edge (l2);
			\draw[edge, dotted]			(t5)--(l3) node[draw = none, fill = none, right = 1mm, midway]{\large $R_3$};
			
                		\node[draw = none, fill = none, below=1mm of l1]  	(a) {\large $a$};
                		\node[draw = none, fill = none, below=1mm of l2]  	(b) {\large $b$};
                		\node[draw = none, fill = none, below=1mm of l3]	(c) {\large $c$};

	 	\end{tikzpicture}
	 	\caption{}
 	\end{subfigure}
 		\begin{subfigure}[b]{.33\textwidth}
		\centering
		\begin{tikzpicture}[edge/.style={very thick, ->}]
	       	
	       		\node[] (C1)	at (0,-5)	{\large $C_1$};
	       	
		       	\node[below left = 0.5 and 0.5 of C1] 	(C2) {\large $C_2$};
		       	\node[below right= 0.5 and 0.5 of C1] 	(R3) {\large $R_3$};
	       		\node[below=1 of C2] 		(R1)	{\large $R_1$};
	       		\node[below=1 of R3]		(R2) {\large $R_2$};
			
			\draw[edge] (C1)--(C2);
			\draw[edge] (C1)--(R3);
			\draw[edge] (C2)--(R1);
			\draw[edge] (C2)--(R2);
			\draw[edge] (R2)--(R3);
			\draw[edge, <->] (R1)--(R2);
			
	 	\end{tikzpicture}
	 	\caption{}
	 	\label{subfig:AuxGraph}
 	\end{subfigure}
    \caption{(a) A semi-binary tree-based network~$N$ that is not orchard on taxa set~$\{a,b,c\}$. 
    A base tree of~$N$ is indicated by the solid edges.
    (b) The bulged version of~$N$ with one possible cherry cover~$\{C_1,C_2,R_1,R_2,R_3\}$.
    Each cherry shape is indicated by the line type used for the edges.
    (c) An auxiliary graph that shows the order on the cherry shapes.
    An edge is drawn from one cherry shape to another if it is directly above it.
    In this case, the cherry cover is not acyclic since~$\{R_1,R_2\}$ form a cycle.}
    \label{fig:TBCherryCover}
\end{figure}

\paragraph{Orchard networks}
An ordered pair of leaves~$(x,y)$ in a network~$N$ is a \emph{cherry} of~$N$ if~$N$ has a cherry shape whose endpoints are~$x$ and~$y$.
Similarly,~$(x,y)$ is a \emph{reticulated cherry} of~$N$ if~$N$ has a reticulated cherry shape whose endpoints are~$x$ and~$y$ and the parent of~$x$ is a reticulation.
We call~$(x,y)$ a \emph{reducible pair} if it is a cherry or a reticulated cherry.
Given an ordered pair of leaves~$(x,y)$, we \emph{reduce}~$(x,y)$ from a graph~$N$ by 
\begin{itemize}
	\item reducing the corresponding cherry shape from~$N$ if~$(x,y)$ is a cherry in~$N$ (if the parent of~$y$ is a bifurcation, then upon deleting the edges of the cherry shape from~$B(N)$, label the unlabelled outdegree-0 vertex as~$y$);
	\item reducing the corresponding reticulated cherry shape from~$N$ if~$(x,y)$ is a reticulated cherry in~$N$;
	\item doing nothing otherwise.
\end{itemize}
Upon reducing~$(x,y)$ from~$N$, we denote the obtained subgraph by~$N(x,y)$.
For a sequence of ordered pairs~$S$, we denote by~$NS$ the network obtained by successively reducing pairs of~$S$ from~$N$ in order.
Observe that this definition of reducing an ordered pair of leaves is equivalent to the definition presented by Janssen and Murakami~\cite{janssen2018solving}.
\begin{definition}
    A network~$N$ is \emph{orchard} if there exists a sequence of ordered pairs~$S$ such that~$NS$ is a network on a single leaf.
\end{definition}
In other words, a network is orchard if we may successively reduce a cherry or a reticulated cherry to obtain a network on a single leaf.
See Figure~\ref{fig:OrchardCherryCover} for an example of an orchard network, its bulged version, and its acyclic cherry cover.

\begin{obs}\label{obs:ReduceBulgedGeneral}
Let~$(x,y)$ be a reducible pair in~$N$, and let~$p_y$ denote the parent of~$y$ in~$N$.
Let~$A$ denote the cherry shape or the reticulated cherry shape corresponding to the reducible pair~$(x,y)$.
\begin{itemize}
	\item If~$p_y$ is a bifurcation, then~$N(x,y) = B^{-1}(B(N)\setminus A)$;
	\item If~$p_y$ is a multifurcation, then~$N(x,y) = B^{-1}((B(N)\setminus A) \cup \{p_yy\})$.
\end{itemize}
\end{obs}


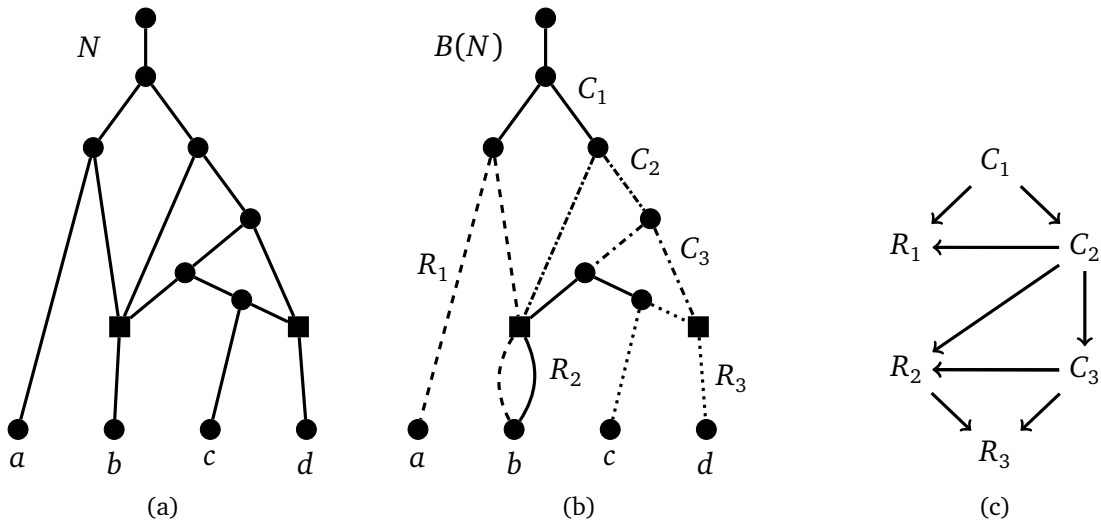
\begin{figure}
	\begin{subfigure}[b]{.33\textwidth}
		\centering
		\begin{tikzpicture}[every node/.style={draw, circle, fill, inner sep=0pt},
		square/.style={regular polygon, regular polygon sides=4}]
	       	\tikzset{edge/.style={very thick}}
	       	
	       		\node[] (root) at (0,5)	{a};
	       	
	       		\node[below = 0.5 of root] 				(t1)	{a};
			\node[below left = 0.75 and 0.5 of t1]		(t2)	{a};
			\node[below right = 0.75 and 0.5 of t1]	(t3)	{a};
			\node[below right = 0.75 and 0.5 of t3]	(t4)	{a};
			
			
			\node[square, below left = 1.2 and 1.5 of t4]	(r1)	{a};
			\node[square, below right = 1.2 and 0.4 of t4]	(r2)	{a};
			
			\node[] (t5) at ($(t4)!0.5!(r1)$) {a};
			\node[] (t6) at ($(t5)!0.5!(r2)$) {a};

			\node[below left = 4.5 and 1.5 of t1]	(l1)	{a};
			\node[right = 1 of l1]				(l2)	{a};
			\node[right = 1 of l2]				(l3)	{a};
			\node[right = 1 of l3]				(l4)	{a};
			
			\draw[edge]		(root) -- (t1) node[draw = none, fill = none, left= 5mm, midway]{\large $N$};
			\draw[edge]		(t1)--(t2);
			\draw[edge]		(t1)--(t3);
			\draw[edge]		(t2)--(r1);
			\draw[edge]		(t3)--(r1);
			\draw[edge]		(t3)--(t4);
			\draw[edge]		(t4)--(t5);
			\draw[edge]		(t4)--(r2);
			\draw[edge]		(t5)--(r1);
			\draw[edge]		(t5)--(t6);
			\draw[edge]		(t6)--(r2);
			
			\draw[edge]		(t2)--(l1);
			\draw[edge]		(r1)--(l2);
			\draw[edge]		(t6)--(l3);
			\draw[edge]		(r2)--(l4);
			
                		\node[draw = none, fill = none, below=1mm of l1]  	(a) {\large $a$};
                		\node[draw = none, fill = none, below=1mm of l2]  	(b) {\large $b$};
	                	\node[draw = none, fill = none, below=1mm of l3]	(c) {\large $c$};
                		\node[draw = none, fill = none, below=1mm of l4]	(d) {\large $d$};
	 	\end{tikzpicture}
	 	\caption{}
 	\end{subfigure}
	\begin{subfigure}[b]{.33\textwidth}
		\centering
		\begin{tikzpicture}[every node/.style={draw, circle, fill, inner sep=0pt},
		square/.style={regular polygon, regular polygon sides=4}]
	       	\tikzset{edge/.style={very thick}}
	       	
	       		\node[] (root) at (0,5)	{a};
	       	
	       		\node[below = 0.5 of root] 				(t1)	{a};
			\node[below left = 0.75 and 0.5 of t1]		(t2)	{a};
			\node[below right = 0.75 and 0.5 of t1]	(t3)	{a};
			\node[below right = 0.75 and 0.5 of t3]	(t4)	{a};
			
			
			\node[square, below left = 1.2 and 1.5 of t4]	(r1)	{a};
			\node[square, below right = 1.2 and 0.4 of t4]	(r2)	{a};
			
			\node[] (t5) at ($(t4)!0.5!(r1)$) {a};
			\node[] (t6) at ($(t5)!0.5!(r2)$) {a};

			\node[below left = 4.5 and 1.5 of t1]	(l1)	{a};
			\node[right = 1 of l1]			(l2)	{a};
			\node[right = 1 of l2]			(l3)	{a};
			\node[right = 1 of l3]			(l4)	{a};
			
			\draw[edge]					(root) -- (t1) node[draw = none, fill = none, left= 5mm, midway]{\large $B(N)$};
			\draw[edge]					(t1)--(t2);
			\draw[edge]					(t1)--(t3) node[draw = none, fill = none, above right = 1mm, midway]{\large $C_1$};
			\draw[edge, dashed]				(t2)--(r1);
			\draw[edge, densely dashdotted]	(t3)--(r1);
			\draw[edge, densely dashdotted]	(t3)--(t4) node[draw = none, above right = 1mm, fill = none, midway]{\large $C_2$};
			\draw[edge, dashdotted]				(t4)--(t5);
			\draw[edge, dashdotted]				(t4)--(r2) node[draw = none, fill = none, above right = 1mm, midway]{\large $C_3$};
			\draw[edge]					(t5)--(r1);
			\draw[edge]					(t5)--(t6);
			\draw[edge, dotted]				(t6)--(r2);
			
			\draw[edge, dashed]				(t2)--(l1) node[draw = none, fill = none, above left = 1mm, midway]{\large $R_1$};
			\draw[edge, bend left]			(r1)edge(l2) node[draw = none, fill = none, above right = 1mm, midway]{\large $R_2$};
			\draw[edge, bend right, dashed]	(r1)edge(l2);
			\draw[edge, dotted]				(t6)--(l3);
			\draw[edge, dotted]				(r2)--(l4) node[draw = none, fill = none, right = 1mm, midway]{\large $R_3$};
			
                		\node[draw = none, fill = none, below=1mm of l1]  	(a) {\large $a$};
                		\node[draw = none, fill = none, below=1mm of l2]  	(b) {\large $b$};
	                	\node[draw = none, fill = none, below=1mm of l3]	(c) {\large $c$};
                		\node[draw = none, fill = none, below=1mm of l4]	(d) {\large $d$};
	 	\end{tikzpicture}
	 	\caption{}
 	\end{subfigure}
	\begin{subfigure}[b]{.33\textwidth}
		\centering
		\begin{tikzpicture}[edge/.style={very thick, ->}]
			
	       		\node[] (C1)	at (0,-5)	{\large $C_1$};
	       	
		       	\node[below left = 0.5 and 0.5 of C1] 	(R1) {\large $R_1$};
		       	\node[below right= 0.5 and 0.5 of C1] 	(C2) {\large $C_2$};
	       		\node[below =1 of R1] 				(R2)	{\large $R_2$};
	       		\node[below =1 of C2]				(C3) {\large $C_3$};
			\node[below right = 0.5 and 0.5 of R2] 	(R3) {\large $R_3$};
			
			\draw[edge] (C1)--(R1);
			\draw[edge] (C1)--(C2);
			\draw[edge] (C2)--(R1);

			\draw[edge] (C2)--(R2);
			\draw[edge] (C2)--(C3);
			\draw[edge] (R2)--(R3);
			\draw[edge] (C3)--(R2);
			\draw[edge] (C3)--(R3);
			\draw[edge] (C3)--(R2);
			
	 	\end{tikzpicture}
		\caption{}
	\end{subfigure}

	\caption{(a) A semi-binary orchard network~$N$ on taxa set~$\{a,b,c,d\}$.
    	One sequence for reducing~$N$ to a network on a single leaf is~$(d,c),(b,a),(b,c),(d,c),(b,c),(a,c)$.
    	(b) The bulged version of~$N$ with one possible cherry cover~$\{C_1,C_2,C_3,R_1,R_2,R_3\}$.
    	(c) An auxiliary graph that shows the order on the cherry shapes.
	In this case, the cherry cover is acyclic.}
	\label{fig:OrchardCherryCover}
\end{figure}


\section{Tree-based networks}

In this section, we show that a binary network is tree-based if and only if it has a cherry decomposition.
We do this by showing that for non-binary networks, the same characterization holds if we look at cherry covers in the bulged version of the network. Taking the bulged version is crucial in this characterization.
Figure~\ref{subfig:JettenLeoNonBulgeCover} (from~\citep{jetten2018nonbinary}) is an example of a (non-bulged) semi-binary network that is not tree-based with a cherry cover.
In the same figure, we show that its bulged version does not have a cherry cover (Figure~\ref{subfig:JettenLeoAuxGraph}), and also that contracting one of the edges in the network yields a non-binary network that is tree-based (Figure~\ref{subfig:JettenLeoTBN}).
This latter point proves the following observation.

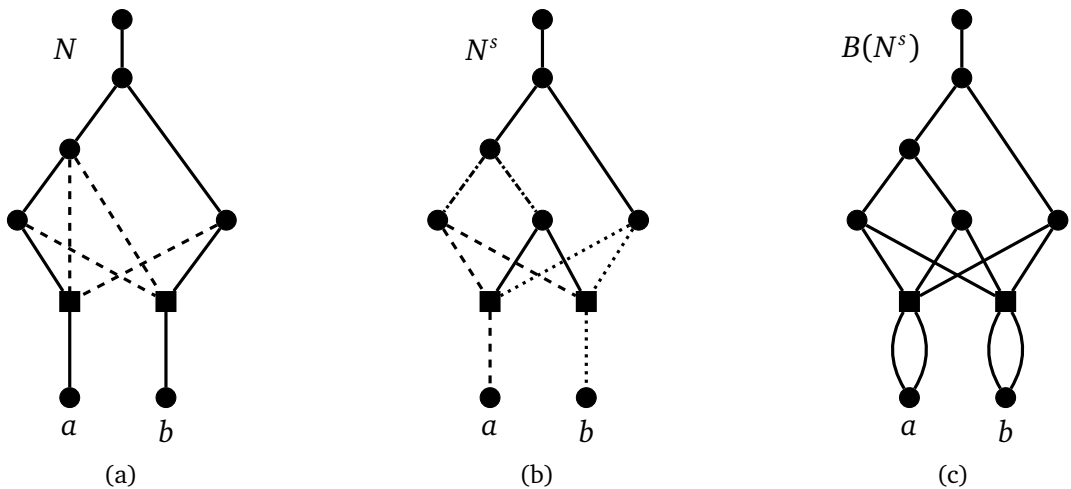
\begin{figure}
	\begin{subfigure}[b]{.33\textwidth}
		\centering
		\begin{tikzpicture}[every node/.style={draw, circle, fill, inner sep=0pt},
		square/.style={regular polygon, regular polygon sides=4}]
	       	\tikzset{edge/.style={very thick}}
	       	
	       		\node[] (root) at (0,5)	{a};
	       	
	       		\node[below = 0.5 of root] 				(t1)	{a};
			\node[below left = 0.75 and 0.5 of t1]		(t2)	{a};
			\node[below left = 0.75 and 0.5 of t2]		(t3)	{a};
			\node[right = 2.5 of t3]				(t5)	{a};
			
			\node[square, below = 1.75 of t2]	(r1)	{a};
			\node[square, right = 1 of r1]		(r2)	{a};

			\node[below = 1 of r1]		(l1)	{a};
			\node[below = 1 of r2]		(l2)	{a};
			
			\draw[edge]			(root) -- (t1) node[draw = none, fill = none, left= 5mm, midway]{\large $N$};
			\draw[edge]			(t1)--(t2);
			\draw[edge]			(t1)--(t5);
			\draw[edge]			(t2)--(t3);
			\draw[edge, dashed]		(t2)--(r1);
			\draw[edge, dashed]		(t2)--(r2);
			
			\draw[edge]			(t3)--(r1);
			\draw[edge, dashed]		(t5)--(r1);
			
			\draw[edge, dashed]		(t3)--(r2);
			\draw[edge]			(t5)--(r2);
			
			\draw[edge]			(r1)--(l1);
			\draw[edge]			(r2)--(l2);
			
                		\node[draw = none, fill = none, below=1mm of l1]  	(a) {\large $a$};
                		\node[draw = none, fill = none, below=1mm of l2]  	(b) {\large $b$};
			
	 	\end{tikzpicture}
	 	\caption{}
	 	\label{subfig:JettenLeoTBN}
 	\end{subfigure}
	\begin{subfigure}[b]{.33\textwidth}
		\centering
		\begin{tikzpicture}[every node/.style={draw, circle, fill, inner sep=0pt},
		square/.style={regular polygon, regular polygon sides=4}]
	       	\tikzset{edge/.style={very thick}}
	       	
	       		\node[] (root) at (0,5)	{a};
	       	
	       		\node[below = 0.5 of root] 				(t1)	{a};
			\node[below left = 0.75 and 0.5 of t1]		(t2)	{a};
			\node[below left = 0.75 and 0.5 of t2]		(t3)	{a};
			\node[below right = 0.75 and 0.5 of t2]	(t4)	{a};
			\node[right = 1 of t4]					(t5)	{a};
			
			\node[square, below = 1.75 of t2]	(r1)	{a};
			\node[square, right = 1 of r1]		(r2)	{a};

			\node[below = 1 of r1]		(l1)	{a};
			\node[below = 1 of r2]		(l2)	{a};
			
			\draw[edge]			(root) -- (t1) node[draw = none, fill = none, left= 5mm, midway]{\large $N^s$};
			\draw[edge]			(t1)--(t2);
			\draw[edge]			(t1)--(t5);
			\draw[edge, densely dashdotted]		(t2)--(t3);
			\draw[edge, densely dashdotted]		(t2)--(t4);
			
			\draw[edge, dashed]		(t3)--(r1);
			\draw[edge]			(t4)--(r2);
			\draw[edge, dotted]		(t5)--(r1);
			
			\draw[edge, dashed]		(t3)--(r2);
			\draw[edge]			(t4)--(r1);
			\draw[edge,dotted]		(t5)--(r2);
			
			\draw[edge, dashed]		(r1)--(l1);
			\draw[edge, dotted]		(r2)--(l2);
			
                		\node[draw = none, fill = none, below=1mm of l1]  	(a) {\large $a$};
                		\node[draw = none, fill = none, below=1mm of l2]  	(b) {\large $b$};
			
	 	\end{tikzpicture}
	 	\caption{}
	 	\label{subfig:JettenLeoNonBulgeCover}
 	\end{subfigure}
	\begin{subfigure}[b]{.33\textwidth}
		\centering
		\begin{tikzpicture}[every node/.style={draw, circle, fill, inner sep=0pt},
		square/.style={regular polygon, regular polygon sides=4}]
	       	\tikzset{edge/.style={very thick}}
	       	
	       		\node[] (root) at (0,5)	{a};
	       	
	       		\node[below = 0.5 of root] 				(t1)	{a};
			\node[below left = 0.75 and 0.5 of t1]		(t2)	{a};
			\node[below left = 0.75 and 0.5 of t2]		(t3)	{a};
			\node[below right = 0.75 and 0.5 of t2]	(t4)	{a};
			\node[right = 1 of t4]					(t5)	{a};
			
			\node[square, below = 1.75 of t2]	(r1)	{a};
			\node[square, right = 1 of r1]		(r2)	{a};

			\node[below = 1 of r1]		(l1)	{a};
			\node[below = 1 of r2]		(l2)	{a};
			
			\draw[edge]			(root) -- (t1) node[draw = none, fill = none, left= 5mm, midway]{\large $B(N^s)$};
			\draw[edge]			(t1)--(t2);
			\draw[edge]			(t1)--(t5);
			\draw[edge]			(t2)--(t3);
			\draw[edge]			(t2)--(t4);
			
			\draw[edge]		(t3)--(r1);
			\draw[edge]		(t4)--(r2);
			\draw[edge]		(t5)--(r1);
			
			\draw[edge]		(t3)--(r2);
			\draw[edge]		(t4)--(r1);
			\draw[edge]		(t5)--(r2);
			
			\draw[edge, bend right]	(r1)edge(l1);
			\draw[edge, bend left]		(r1)edge(l1);
			\draw[edge, bend right]	(r2)edge(l2);
			\draw[edge, bend left]		(r2)edge(l2);
			
                		\node[draw = none, fill = none, below=1mm of l1]  	(a) {\large $a$};
                		\node[draw = none, fill = none, below=1mm of l2]  	(b) {\large $b$};
			
	 	\end{tikzpicture}
	 	\caption{}
	 	\label{subfig:JettenLeoAuxGraph}
 	\end{subfigure}
    \caption{(a) A non-binary tree-based network~$N$ on~$\{a,b\}$. A base tree is indicated by the solid edges.
    (b) A semi-binary resolution~$N^{s}$ of~$N$ that is not tree-based with a cherry cover.
    (c) The bulged version of~$N^{s}$ that does not have a cherry cover.
    This can be seen as follows. 
    There are four edges incident to the leaves and each of these edges can only be covered by reticulated cherry shapes. 
    However, it is not possible to add four such reticulated cherry shapes without covering any middle edge of a reticulated cherry shape more than once.}
    \label{fig:JettenLeo}
\end{figure}

\begin{obs}
    Let~$N$ be a tree-based network.
    Then there may exist a semi-binary resolution of~$N$ that is not tree-based.
\end{obs}
\begin{lem}\label{lem:TBiffResTB}
    Let $N$ be a network. Then $N$ is tree-based if and only if some binary resolution of $N$ is tree-based. 
    $N$ is tree-based if and only if there some semi-binary resolution of~$N$ is tree-based. 
\end{lem}
\begin{proof}
    The first statement follows from \cite{jetten2018nonbinary} Observation~3.2.
    To show the second statement, let~$N$ be a tree-based network, and let~$T$ be a base tree of~$N$.
    By definition of base trees,~$T$ must visit every tree vertex in the network.
    In particular, it must visit every multifurcation, and exit such vertices via one of its outgoing edges.
    Let~$t$ denote such a tree vertex and let~$s$ denote the child of~$t$ in~$N$ such that~$ts$ is an edge that is used by~$T$.
    Then we resolve~$t$ by replacing it by a caterpillar such that the parent of~$s$ is the bottom-most vertex.
    It remains to check that the base tree covers all the newly introduced vertices.
    However this is immediate; by the placement of~$s$, we note that the path from~$t$ to~$s$ covers all the newly introduced vertices.
    Therefore the tree~$T$ with the edge~$ts$ changed to the path from~$t$ to~$s$ is a base tree of the new network.
    Repeating this for all multifurcations yields a semi-binary resolution of~$N$ that is tree-based.

    On the other hand, if a semi-binary resolution~$N'$ of~$N$ is tree-based, then it is easy to see that~$N$ must also be tree-based.
    Indeed, upon contracting some of the edges of~$N'$, we adjust the base tree of~$N'$ by contracting the same edge in the base tree if it used that edge in the embedding, and not changing the base tree otherwise.
    Doing so gives a base tree of~$N$.
\end{proof}

\begin{thm}\label{thm:CherryDecompositiontree-based}
    A network~$N$ is tree-based if and only if~$B(N)$ has a cherry cover.
\end{thm}
\begin{proof}
    First suppose that~$N$ is a tree-based network.
    Let~$T$ be a base tree of~$N$, and let~$E_r = \{e_1,\ldots, e_k\}$ denote the reticulation edges that were deleted to obtain~$T$ from~$N$.
    By Lemma~\ref{lem:OnlyCherryCoverThenTree}, $T$ has a cherry cover $P$ consisting of only cherry shapes. 
    We use this cherry cover to produce a cherry cover of $N$.
    
    Each cherry shape $C$ in $P$ maps to a pair of paths $c_1$ and $c_2$ in $B(N)$ that are vertex-disjoint except at their highest vertices. 
    All these paths together cover the edges of the embedding $E_T$ of $T$ in $B(N)$.
    Taking the first arc of both $c_1$ and $c_2$, we obtain a cherry shape $C|_N$ of $B(N)$. 
    Let $P'=\{C|_N:C\in P\}$ be the set of cherry shapes in $B(N)$ obtained from cherry shapes in $P$, and let $F=E_T \setminus P'$ be the edges of~$E_T$ that are not covered by $P'$.
    
    The edges of $B(N)$ apart from the root edge that are not yet covered by $P'$ are as follows:
    \begin{itemize}
        \item the reticulation edges $e_i = u_i v_i \in E_r$,
        \item all outgoing edges of~$v_i$ for all~$i\in[k]$,
        \item and for each $u_i$ for all~$i\in [k]$, at most one outgoing edge~$f_{u_i}\in F$.
    \end{itemize}
    For the last point, if the endpoint $u_i$ were to have more than one outgoing edges in $F$, then they would be part of a cherry shape in $P'$; hence, they cannot be in $F$, but they must be in $P'$.
    Therefore this case is not possible.
    If there is no outgoing edge of $u_i$ contained in $F$, then~$u_i$ must have two outgoing edges that form a cherry shape in~$B(N)$ that is contained in~$P'$.
    Otherwise~$u_i$ would not have been covered by~$E_T$, which would contradict the fact that~$T$ was a base tree of~$B(N)$.
    If there was exactly one outgoing edge $f_{u_i}$ of $u_i$ contained in $F$, then $u_i$ was not a tree vertex in~$T$ (in particular it must have been added as an attachment point). 
    Thus, $f_{u_i}$ is not a highest arc in the embedding of a cherry shape of $P$, so $f_{u_i}$ is not covered by $P'$. 
    Observe that~$f_{u_i}$ cannot be the reticulation edge~$e_i$ itself, since~$E_r$ contains all the reticulation edges that are not used in the embedding of~$T$ in~$N$.
    Therefore, each endpoint~$u_i$ of a reticulation edge $e_i = u_iv_i\in E_r$ has an outgoing edge in $F$, or an outgoing edge that is covered by $P'$.
   
    We augment $P'$ to a cherry cover $P''$ of $B(N)$ by adding a reticulated cherry shape $\{v_ix_i,u_iv_i,u_iy_i\}$ for each $e_i=u_iv_i\in E_r$ satisfying the following conditions: for each $i$, either $u_iy_i\in F$ or $u_iy_i$ is covered by $P'$, and for any $i \neq j$, $v_ix_i\neq v_jx_j$.
    This last condition is possible because the number of outgoing edges of a reticulation vertex $v$ is equal to the number of incoming arcs of $v$ that are in $E_r$.
    By construction, $P''$ is a cherry cover of $B(N)$.
    
        \medskip
    
    Now suppose that the bulged version of the network~$N$ has a cherry cover~$P$.
    For every reticulation vertex~$v$ of indegree~$k$, exactly~$k-1$ incoming edges are contained in a reticulated cherry shape as the middle edge in~$P$.
    By definition of reticulated cherry shapes, the tail of each of these reticulation edges has at least one child other than~$v$.
    This inherently implies that deleting these~$k-1$ reticulation edges will not create any unlabelled outdegree-0 vertices.
    Repeating this deletion for all such reticulation edges and removing all parallel edges returns a spanning tree of the graph whose leaves are labelled bijectively by the leaf-set of~$N$; therefore $B(N)$ is tree-based.
    By Observation~\ref{obs:BulgeOfTBisTB},~$N$ is tree-based.
 \end{proof}

By Lemma~\ref{lem:OnlyCherryCoverThenTree}, there exists a cherry cover of a tree on~$n$ leaves that contains exactly~$n-1$ cherry shapes.
The next corollary follows immediately from this observation and the arguments used in the proof of Theorem~\ref{thm:CherryDecompositiontree-based}.

\begin{cor}\label{cor:CherryCoverCount}
	Let $N$ be a tree-based network on~$n$ leaves and reticulation number~$r$.
	Then there exists a cherry cover of~$N$ that contains exactly~$n-1$ cherry shapes and exactly~$r$ reticulated cherry shapes.
\end{cor}



\section{Orchard networks}

We now show that a binary network is orchard if and only if it has an acyclic cherry decomposition.
Like in the previous section, we do this by showing that a non-binary network is orchard if and only if the bulged version of the network has an acyclic cherry cover.
Note that this added condition of taking the bulged version of the network is necessary.
There exist networks that are not orchard with an acyclic cherry cover (see Figure~\ref{fig:NotOrchAcyclic}).

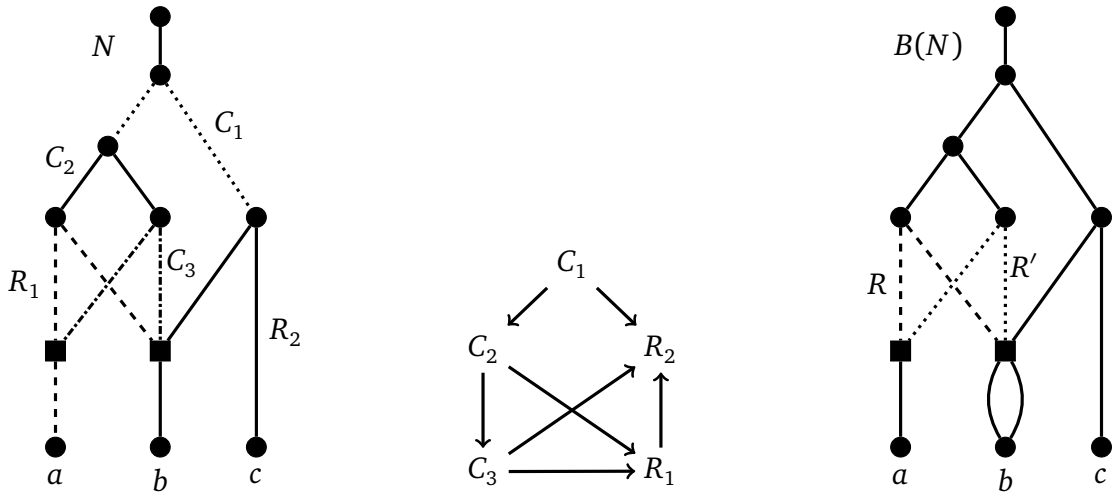
\begin{figure}[h!]
	\begin{subfigure}[b]{.33\textwidth}
		\centering
		\begin{tikzpicture}[every node/.style={draw, circle, fill, inner sep=0pt},
		square/.style={regular polygon, regular polygon sides=4}]
	       	\tikzset{edge/.style={very thick}}
	       	
	       	\node[] (root) at (0,5)	{a};
	       	
	       	\node[below = 0.5 of root] 				(t1)	{a};
			\node[below left = 0.75 and 0.5 of t1]		(t2)	{a};
			\node[below left = 0.75 and 0.5 of t2]		(t3)	{a};
			\node[below right = 0.75 and 0.5 of t2]	(t4)	{a};
			\node[right = 1 of t4]					(t5)	{a};
			
			\node[square, below = 1.5 of t3]	(r1)	{a};
			\node[square, below = 1.5 of t4]	(r2)	{a};

			\node[below = 1 of r1]		(l1)	{a};
			\node[below = 1 of r2]		(l2)	{a};
			\node[right = 1 of l2]			(l3)	{a};
			
			\draw[edge]			(root) -- (t1)node[draw = none, fill = none, left= 5mm, midway]{\large $N$};
			\draw[edge, dotted]			(t1)--(t2);
			\draw[edge, dotted]			(t1)--(t5) node[draw = none, fill = none, above right = 1mm, midway]{\large $C_1$};
			\draw[edge]			(t2)--(t3) node[draw = none, fill = none, above left = 1mm, midway]{\large $C_2$};
			\draw[edge]			(t2)--(t4);
			
			\draw[edge, dashed]			(t3)--(r1) node[draw = none, fill = none, left = 1mm, midway]{\large $R_1$};
			\draw[edge, densely dashdotted]			(t4)--(r2) node[draw = none, fill = none, above right = 1mm, midway]{\large $C_3$};
			
			\draw[edge, dashed]			(t3)--(r2);
			\draw[edge, densely dashdotted]			(t4)--(r1);
			\draw[edge]			(t5)--(r2);
			
			\draw[edge, dashed]			(r1)--(l1);
			\draw[edge]			(r2)--(l2);
			\draw[edge]			(t5)--(l3) node[draw = none, fill = none, right = 1mm, midway]{\large $R_2$};
			
                	\node[draw = none, fill = none, below=1mm of l1]  	(a) {\large $a$};
                	\node[draw = none, fill = none, below=1mm of l2]  	(b) {\large $b$};
                	\node[draw = none, fill = none, below=1mm of l3]	(c) {\large $c$};
			
	 	\end{tikzpicture}
 	\end{subfigure}
	\begin{subfigure}[b]{.33\textwidth}
		\centering
		\begin{tikzpicture}[edge/.style={very thick, ->}]
	       	
	       		\node[] (C1)	at (0,-5)	{\large $C_1$};
	       	
		       	\node[below left = 0.5 and 0.5 of C1] 	(C2) {\large $C_2$};
		       	\node[below right= 0.5 and 0.5 of C1] 	(R2) {\large $R_2$};
	       		\node[below=1 of C2] 		(C3)	{\large $C_3$};
	       		\node[below=1 of R2]		(R1) {\large $R_1$};
			
			\draw[edge] (C1)--(C2);
			\draw[edge] (C1)--(R2);
			\draw[edge] (C2)--(C3);
			\draw[edge] (C2)--(R1);
			\draw[edge] (C3)--(R1);
			\draw[edge] (C3)--(R2);
			\draw[edge] (R1)--(R2);
			
	 	\end{tikzpicture}
	\end{subfigure}
		\begin{subfigure}[b]{.33\textwidth}
		\centering
		\begin{tikzpicture}[every node/.style={draw, circle, fill, inner sep=0pt},
		square/.style={regular polygon, regular polygon sides=4}]
	       	\tikzset{edge/.style={very thick}}
	       	
	       	\node[] (root) at (0,5)	{a};
	       	
	       	\node[below = 0.5 of root] 				(t1)	{a};
			\node[below left = 0.75 and 0.5 of t1]		(t2)	{a};
			\node[below left = 0.75 and 0.5 of t2]		(t3)	{a};
			\node[below right = 0.75 and 0.5 of t2]	(t4)	{a};
			\node[right = 1 of t4]					(t5)	{a};
			
			\node[square, below = 1.5 of t3]	(r1)	{a};
			\node[square, below = 1.5 of t4]	(r2)	{a};

			\node[below = 1 of r1]		(l1)	{a};
			\node[below = 1 of r2]		(l2)	{a};
			\node[right = 1 of l2]			(l3)	{a};
			
			\draw[edge]			(root) -- (t1) node[draw = none, fill = none, left= 5mm, midway]{\large $B(N)$};
			\draw[edge]			(t1)--(t2);
			\draw[edge]			(t1)--(t5);
			\draw[edge]			(t2)--(t3);
			\draw[edge]			(t2)--(t4);
			
			\draw[edge, dashed]			(t3)--(r1) node[draw = none, fill = none, left = 1mm, midway]{\large $R$};
			\draw[edge, dotted]			(t4)--(r2) node[draw = none, fill = none, above right = 1mm, midway]{\large $R'$};
			
			\draw[edge, dashed]			(t3)--(r2);
			\draw[edge, dotted]			(t4)--(r1);
			\draw[edge]			(t5)--(r2);
			
			\draw[edge]				(r1)--(l1);
			\draw[edge, bend left]		(r2)edge(l2);
			\draw[edge, bend right]	(r2)edge(l2);
			\draw[edge]				(t5)--(l3);
			
                	\node[draw = none, fill = none, below=1mm of l1]  	(a) {\large $a$};
                	\node[draw = none, fill = none, below=1mm of l2]  	(b) {\large $b$};
                	\node[draw = none, fill = none, below=1mm of l3]	(c) {\large $c$};
			
	 	\end{tikzpicture}
 	\end{subfigure}

	\caption{An example showing why we consider cherry covers in bulged versions of networks.
	The tree-based network~$N$ from Figure~\ref{fig:TBCherryCover} (a).
	We obtain a cherry cover~$\{C_1,C_2,C_3,R_1,R_2\}$ of~$N$ that is acyclic, although the network is not orchard.
	The bulged version of~$N$ that shows that every cherry cover of~$N$ must be cyclic.
	Indeed, the edges labelled~$R$ and~$R'$ must be contained in a reticulated cherry shape whose endpoint is the leaf~$a$ or~$b$.
	These shapes will be directly above one another, creating a cycle in the auxiliary graph.}
	\label{fig:NotOrchAcyclic}
\end{figure}

In Figure~\ref{fig:NotOrchAcyclic}, observe that any cherry cover of the bulged version of the network is not acyclic.
The edge incident to~$a$ must be covered by a reticulated cherry shape -- say it is covered by a reticulated cherry shape~$R$.
In the bulged version of the network, there are parallel edges incident to the leaf~$b$; one of these edges must be covered by a reticulated cherry shape containing the edges of~$R'$.
But~$R$ and~$R'$ are above one another, and therefore no cherry cover can be acyclic.

In Figure~\ref{fig:OrchardCE}, the network~$N$ is an orchard network, as~$(a,b)(d,c)(b,c)(a,c)(d,c)$ is a sequence of reducible pairs that reduce~$N$ to a network on a single leaf~$c$.
In the same figure, a semi-binary resolution~$N^s$ of~$N$ that is not orchard is presented.
Since there are no reducible pairs (no cherries nor reticulated cherries) in~$N^s$, it is immediately clear that~$N^s$ is not orchard.
Therefore we obtain the following observation.

\begin{obs}
    Let $N$ be an orchard network. 
    Then there may exist a semi-binary resolution of $N$ that is not orchard.
\end{obs}

\begin{figure}
	\begin{subfigure}[b]{.5\textwidth}
		\centering
		\begin{tikzpicture}[every node/.style={draw, circle, fill, inner sep=0pt},
		square/.style={regular polygon, regular polygon sides=4}]
	       	\tikzset{edge/.style={very thick}}
	       	
	       	\node[] (root) at (0,5)	{a};
	       	
	       	\node[below = 0.5 of root] 				(t1)	{a};
			\node[below left = 0.5 and 0.5 of t1]		(t2)	{a};
			\node[below right = 0.5 and 0.5 of t2]		(t3)	{a};
			
			\node[square, below left = 2.5 and 1 of t1]	(r1)	{a};
			\node[square, below right = 2.5 and 1 of t1]	(r2)	{a};
			
			\node[below left = 3.25 and 1.5 of t1] 		(l1) {a};
			\node[below left = 3.25 and 0.5 of t1] 		(l2) {a};
			\node[below right = 3.25 and 0.5 of t1] 	(l3) {a};
			\node[below right= 3.25 and 1.5 of t1] 		(l4) {a};

			\draw[edge]			(root) -- (t1) node[draw = none, fill = none, left= 5mm, midway]{\large $N$};
			\draw[edge]			(t1)--(t2);
			\draw[edge]			(t1)--(r2);
			\draw[edge]			(t2)--(r1);
			\draw[edge]			(t2)--(t3);
			\draw[edge]			(t3)--(r1);
			\draw[edge]			(t3)--(r2);
			
			\draw[edge]			(r1)--(l1);
			\draw[edge]			(t3)--(l2);
			\draw[edge]			(t3)--(l3);
			\draw[edge]			(r2)--(l4);
			
                	\node[draw = none, fill = none, below=1mm of l1]  	(a) {\large $a$};
                	\node[draw = none, fill = none, below=1mm of l2]  	(b) {\large $b$};
                	\node[draw = none, fill = none, below=1mm of l3]	(c) {\large $c$};
                	\node[draw = none, fill = none, below=1mm of l4]	(c) {\large $d$};
			
	 	\end{tikzpicture}
 	\end{subfigure}
	\begin{subfigure}[b]{.5\textwidth}
		\centering
		\begin{tikzpicture}[every node/.style={draw, circle, fill, inner sep=0pt},
		square/.style={regular polygon, regular polygon sides=4}]
	       	\tikzset{edge/.style={very thick}}
	       	
	       	\node[] (root) at (0,5)	{a};
	       	
	       	\node[below = 0.5 of root] 				(t1)	{a};
			\node[below left = 0.5 and 0.5 of t1]		(t2)	{a};
			\node[below = 1 of t1]	(t3) {a};
			\node[below = 1.5 of t1]	(t4) {a};
			\node[below = 2 of t1]	(t5) {a};
			
			\node[square, below left = 2.5 and 1 of t1]	(r1)	{a};
			\node[square, below right = 2.5 and 1 of t1]	(r2)	{a};
			
			\node[below left = 3.25 and 1.5 of t1] 		(l1) {a};
			\node[below left = 3.25 and 0.5 of t1] 		(l2) {a};
			\node[below right = 3.25 and 0.5 of t1] 	(l3) {a};
			\node[below right= 3.25 and 1.5 of t1] 		(l4) {a};

			\draw[edge]			(root) -- (t1) node[draw = none, fill = none, left= 5mm, midway]{\large $N^s$};
			\draw[edge]			(t1)--(t2);
			\draw[edge]			(t1)--(r2);
			\draw[edge]			(t2)--(r1);
			\draw[edge]			(t2)--(t3);
			\draw[edge]			(t5)--(r1);
			\draw[edge]			(t5)--(r2);
			\draw[edge]			(t3)--(t5);
			
			\draw[edge]			(r1)--(l1);
			\draw[edge]			(t3)--(l2);
			\draw[edge]			(t4)--(l3);
			\draw[edge]			(r2)--(l4);
			
                	\node[draw = none, fill = none, below=1mm of l1]  	(a) {\large $a$};
                	\node[draw = none, fill = none, below=1mm of l2]  	(b) {\large $b$};
                	\node[draw = none, fill = none, below=1mm of l3]	(c) {\large $c$};
                	\node[draw = none, fill = none, below=1mm of l4]	(c) {\large $d$};
			
	 	\end{tikzpicture}
	\end{subfigure}
    \caption{An orchard network $N$ and a non-orchard semi-binary resolution $N^s$ of~$N$.}
    \label{fig:OrchardCE}
\end{figure}
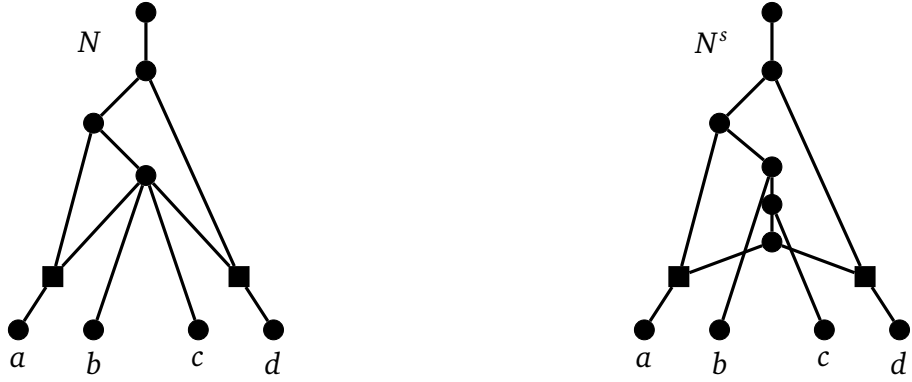

\begin{lem}\label{lem:ReducingLowestShapeGeneral}
    Let $N$ be a network where~$B(N)$ has a cherry cover $P$. 
    Suppose $A\in P$ is a lowest shape with endpoints $x$ and $y$ and an edge $p_yy$, where~$p_y$ is a tree vertex. 
    Then,
    \begin{itemize}
        \item $(x,y)$ is a reducible pair in $N$, and
        \item $B(N(x,y))$ has a cherry cover~$P\setminus \{A\}$ if~$p_y$ is a bifurcation; otherwise,~$B(N(x,y))$ has a cherry cover~$(P\setminus \{A\}) \cup\{Z\}$, where $Z$ is a shape with endpoint $y$.
    \end{itemize}
\end{lem}
\begin{proof}
    We first show that $x$ and $y$ are leaves in $B(N)$.
    Suppose for a contradiction that $x$ is not a leaf.
    Then it is either a tree vertex or a reticulation vertex. 
    In either case, $x$ has an outgoing edge which must be part of some shape $Y\in P$. 
    As $x$ is not the lowest vertex in this shape~$Y$, $x$ must be an internal vertex of $Y$. 
    This implies that $A$ is above $Y$, which contradicts the fact that $A$ is a lowest shape. 
    Hence, $x$ must be a leaf. By the same argument, $y$ is a leaf. 
    As the vertex labels do not change when considering bulged versions of networks, $x$ and $y$ are also leaves in $N$.
    We now split into two cases: either $A$ is a cherry shape, or $A$ is a reticulated cherry shape.

    First suppose $A=\{p_yx,p_yy\}$ is a cherry shape. 
    As $B(N)$ has edges $p_yx$ and $p_yy$, $N$ must also have such edges. 
    As $N$ has edges $p_yx$ and $p_yy$, and $x$ and $y$ are leaves, $N$ has the cherry $(x,y)$. This means $(x,y)$ is a reducible pair in $N$.
    Now suppose $A=\{p_xx,p_yp_x,p_yy\}$ is a reticulated cherry shape. 
    Then in $N$, there are also edges $p_xx$, $p_yp_x$, and $p_yy$. 
    As $x$ and $y$ are leaves in $N$ and $p_x$ is a reticulation vertex---by the properties of a cherry cover---$(x,y)$ is a reticulated cherry in $N$, which is a reducible pair.
    This proves the first part of the lemma. 
    
    
    For the second part of the lemma, we split the proof into two subcases.
    First suppose that $p_y$ is a bifurcation. 
    By assumption, $P$ is a cherry cover of $B(N)$, $A$ is an element of $P$, and by Observation~\ref{obs:ReduceBulgedGeneral}, we have $N(x,y)=B^{-1}(B(N)\setminus A)$.
    It follows that the set $P\setminus\{A\}$ is a cherry cover of $B(N(x,y))= B(B^{-1}(B(N)\setminus A))=B(N)\setminus A$.
    
    Now suppose that $p_y$ is a multifurcation. 
    Then, $B(N(x,y))=(B(N)\setminus A)\cup\{p_yy\}$ by Observation~\ref{obs:ReduceBulgedGeneral} and $P\setminus\{A\}$ covers all edges of $B(N)\setminus A$, so only the edge $p_yy$ may not be covered by~$P\setminus\{A\}$. 
    If the edge~$p_yy$ is covered by $P\setminus\{A\}$, then this is a cherry cover of $(B(N)\setminus A)\cup\{p_yy\}$ and therefore of $B(N(x,y))$ and we are done. 
    So suppose $p_yy$ is not covered by $P\setminus \{A\}$.
    Excluding the edge~$p_yy$, if all other outgoing edges of~$p_y$ formed the middle edge of reticulated cherry shapes, then~$p_yy$ must have formed the free edge of each of these reticulated cherry shapes.
    This implies that the edge~$p_yy$ must already have been covered by~$P\setminus \{A\}$, which is not true by our assumption.
    Therefore, there exists some outgoing edge~$p_yz$ of~$p_y$ that is covered by~$P\setminus\{A\}$, such that~$p_yz$ does not form the middle edge of a reticulated cherry shape. 
    Then, we obtain a cherry cover~$(P\setminus\{A\}) \cup \{p_yy,p_yz\}$ of~$(B(N)\setminus A) \cup \{p_yy\}$ and therefore of~$B(N(x,y))$.
\end{proof}

\begin{thm}\label{thm:CherryDecompositionorchard}
    A network $N$ is orchard if and only if $B(N)$ has an acyclic cherry cover.
\end{thm}
\begin{proof}
    First suppose that a network~$N$ is orchard.
    We prove by induction on the sum $S=n+r$ of the number of leaves $n$ and the reticulation number $r$ of~$N$ that $B(N)$ has an acyclic cherry cover. 
    The induction basis is the case with one leaf and no reticulations: the empty set is an acyclic cherry cover for such a network. 
    
    Now suppose that for each orchard network $N'$ with $n'+r'=S'$, $B(N')$ has an acyclic cherry cover. 
    We prove that for any network $N$ with $n+r=S=S'+1$, $B(N)$ has an acyclic cherry cover. 
    For this purpose, let $N$ be an orchard network with $n$ leaves and $r$ reticulations, such that $n+r=S=S'+1$, and let $(x,y)$ be a reducible pair in $N$. 
    Note that as~$N$ is an orchard network, such a reducible pair must exist.
    
    First suppose that the parent $p_y$ of $y$ is a bifurcation. 
    By Observation~\ref{obs:ReduceBulgedGeneral}, we have that $B(N(x,y))=B(N)\setminus A$, where $A$ is a cherry shape if $(x,y)$ is a cherry, and $A$ is a reticulated cherry shape if $(x,y)$ is a reticulated cherry. 
    By definition of orchard networks and reductions of reducible pairs,~$N(x,y)$ is an orchard network and the sum of its leaves and reticulations is $S'$. 
    By the induction hypothesis,~$B(N(x,y))$ has an acyclic cherry cover~$P$.
    We may obtain a cherry cover for~$B(N)$ by appending the shape~$A$ to~$P$.
    Therefore~$B(N)$ has a cherry cover~$P\cup\{A\}$.
    As the endpoints of~$A$ are leaves, the element~$A$ is above no other shape in~$P$.
    Therefore the cherry cover~$P\cup\{A\}$ is acyclic.
    
    Now suppose that the parent of $p_y$ is a multifurcation. 
    By Observation~\ref{obs:ReduceBulgedGeneral}, $B(N(x,y))=(B(N)\setminus A)\cup \{p_yy\}$, where $A$ is either a cherry shape or a reticulated cherry shape on $(x,y)$. 
    We have again that $N(x,y)$ is an orchard network, and the sum of its leaves and reticulations is $S'$. 
    By the induction hypothesis, this implies that there is an acyclic cherry cover $P$ of $B(N(x,y))\cup \{p_yy\}$, which gives a cherry cover $P\cup\{A\}$ of $B(N)$. 
    This cherry cover is acyclic because the new element $A$ is above no other shape as its endpoints are leaves. 
    
    Hence, for each orchard network~$N$ with a total $S'+1$ of leaves and reticulations, there is an acyclic cherry cover of $B(N)$.
    \medskip

    To prove the other direction of the theorem, suppose that $B(N)$ has an acyclic cherry cover $P$ and let $A\in P$ be a lowest shape with endpoints $x$ and $y$.
    Observe that such a lowest shape must exist as otherwise the cherry cover would not be acyclic.
    By Lemma~\ref{lem:ReducingLowestShapeGeneral}, $(x,y)$ is a reducible pair in $B(N)$, and $B(N(x,y))$ has a cherry cover $(P\setminus \{A\}) \cup\{Z\}$ or $P$, in which the order on the remaining shapes is not changed. 
    This means $B(N(x,y))$ is smaller than $B(N)$, and it has an acyclic cherry cover. 
    This process continues until $P=\emptyset$, and both $N$ and $B(N)$ are reduced to a single leaf network.
    Since we have successively reduced cherries or reticulated cherries from~$N$ to obtain a single leaf network, $N$ is an orchard network.
\end{proof}

We now prove a lemma that is analogous to Lemma~\ref{lem:TBiffResTB} for orchard networks using acyclic cherry covers.

\begin{lem}\label{lem:OrchardResolution}
    Let $N$ be a network.
    Then~$N$ is orchard if and only if some binary resolution of~$N$ is orchard.
    $N$ is orchard if and only if some semi-binary resolution of~$N$ is orchard.
\end{lem}
\begin{proof}
	We first assume that there exists some binary resolution~$N^b$ of~$N$ that is orchard, and independently, that there exists some semi-binary resolution~$N^s$ of~$N$ that is orchard.
	We claim that contracting an edge of an orchard network whose head and tail are both tree vertices or both reticulation vertices returns an orchard network.
	By definition, we may obtain~$N$ by contracting exactly these edges from~$N^b$ and from~$N^s$ (different edges for the two resolutions), from which it follows that~$N$ is orchard.
	We now prove the claim.
	
	Let~$M$ be an orchard network, and let~$st$ be an edge in~$M$ such that~$s$ and~$t$ are both tree vertices.
	We show that the network obtained by contracting~$st$ in~$M$ is orchard.
	By Theorem~\ref{thm:CherryDecompositionorchard},~$M$ has an acyclic cherry cover~$P$.
	The edge~$st$ is covered as an edge in a cherry shape or as a free edge in a reticulated cherry shape in~$P$ (or possibly both and multiple times, in case~$s$ is a multifurcation).
	Moreover, at least one of the outgoing edges of~$t$ is also covered as an edge in a cherry shape or as a free edge in a reticulated cherry shape in~$P$.
	Let us denote this edge by~$tu$.
	We now contract the edge~$st$, and replace the edge~$st$ that appeared in every shape in~$P$ by~$tu$.
	All other shapes of~$P$ are preserved and we call this new set of shapes~$P'$. 
	All edges of the contracted network are covered and it is easy to check that~$P'$ is a cherry cover.
	It remains to show that~$P'$ is an acyclic cherry cover.
	However this follows immediately.
	Indeed, the shapes in~$P$ that contained the edge~$st$ are no longer directly above the shapes in~$P$ that contained the vertex~$t$ as an internal vertex; furthermore, the shapes in~$P$ that contained the edge~$st$ are now directly above the shapes in~$P$ that contained the vertex~$u$ as an internal vertex.
	These new edges do not create a cycle in the auxiliary graph, as otherwise~$P$ would have been cyclic.
	
	Now let~$pq$ be an edge in~$M$ such that~$p$ and~$q$ are both reticulations.
	By definition of cherry covers, there must exist one incoming edge~$kp$ of~$p$ such that~$kp$ is covered as an edge in a cherry shape or as a free edge in a reticulated cherry shape in~$P$.
	Call this shape~$A$. 
	Let~$r$ be a child of~$q$.
	We now contract the edge~$pq$, and replace the edge~$pq$ that appeared in every shape in~$P$ by~$qr$.
	We preserve all other shapes of~$P$ and we call this new set of shapes~$P'$.
	Observe that~$P'$ forms a cherry cover of the contracted network, and that it is acyclic.
	The only difference between the auxiliary graph of~$P$ and the auxiliary graph of~$P'$ is that the arrow between shapes of~$P$ containing~$pq$ and the shapes of~$P$ containing~$qr$ has been deleted, and the arrow from~$A$ to shapes of~$P$ containing~$qr$ has been added.
	But~$A$ was already above these shapes in the auxiliary graph of~$P$.
	The same can be said for all reticulated cherry shapes that covered an incoming edge of~$p$ as the middle edge.
	This implies that the new cherry cover must be acyclic.
	
	Hence, contracting an edge of an orchard network whose head and tail are both tree vertices or both reticulation vertices returns an orchard network.
	Therefore the network~$N$ is orchard.
	\medskip
	
	To prove the other direction, suppose that~$N$ is an orchard network.
	By Lemma~$2$ in~\cite{janssen2018solving}, there exists a binary resolution of~$N$ that is orchard.
	Since a binary network is semi-binary, we are done.
\end{proof}

It was shown in~\cite{huber2019rooting} that binary orchard networks are tree-based.
It follows from Theorems~\ref{thm:CherryDecompositiontree-based} and~\ref{thm:CherryDecompositionorchard} that this is also true for the non-binary case.

\begin{cor}\label{cor:NBOisNBTB}
	All orchard networks are tree-based.
\end{cor}

\section{Discussion}

In this paper we have provided a unifying structural characterization for 
\leoo{tree-based networks and orchard networks} using cherry covers.
We have shown that a binary network is tree-based if and only if it can be decomposed into cherry shapes and reticulated cherry shapes. 
A binary network is orchard if such a decomposition exists that also satisfies a certain acyclicity condition. 
Moreover, we have generalized these characterizations to non-binary networks by considering bulged versions of the networks and using covers rather than decompositions.
Prior to having this characterization, orchard networks were characterized only by the sequences that reduced them.
Therefore we have provided the first structural \leoo{(non-recursive)} characterization for orchard networks.

Structural characterizations for many network classes have generally focused more on `forbidden structures' rather than on decompositions.
Tree-based networks cannot contain a maximum zig-zag path that starts and ends at a reticulation (\emph{W-shapes}); tree-child networks cannot contain adjacent reticulation vertices nor tree vertices with only reticulation children.
While structures such as crowns (a bipartite graph between some subset of the tree vertices and reticulations that contains an undirected cycle) and W-shapes cannot be contained in orchard networks, it remains open whether orchard networks can be characterized by a list of forbidden substructures. 


In the other direction, it may be of interest to extend our cherry cover results to characterize other network classes that are contained in the class of tree-based networks.
Since (the bulged versions of) these networks have a cherry cover, this may be possible by imposing additional conditions on the cherry cover.
\leoo{Finding such} characterization\leoo{s} for all known network classes, such as tree-child, reticulation-visible, and stack-free, will truly bring to light a \emph{unifying} structural characterization of \leoo{known} phylogenetic network \leoo{classes}.

Outside of characterizing network classes, cherry covers can be used to prove other results within phylogenetics.
One particular case in which this could have been useful is in the setting of chain reductions as done in the paper~\cite{huber2019rooting}.
In that paper, it was shown that leaves may be added to, and some leaves may be removed from orchard and tree-based networks to obtain a network that was still orchard or tree-based (in particular, Lemmas~$10, 11$ and~$13$).
Employing cherry covers to prove these results is more concise, since adjusting the cherry cover of networks after such actions is easier than trying to alter, say, the sequence for the network (for orchards).

Another area in which cherry covers may be useful is in solving enumeration problems, which is formulated as follows.
Given parameters~$n$ and~$r$, find the number of distinct networks on~$n$ leaves with reticulation number~$r$.
There \leoo{exist} cherry covers for such networks that contain~$n-1$ cherry shapes and~$r$ reticulated cherry shapes by Corollary~\ref{cor:CherryCoverCount}---can we somehow count all possible arrangements of these shapes to enumerate the space of both network classes?
Perhaps, for non-binary networks, this line of attack will be too complicated due to shapes being able to cover the same edges.
However, for binary networks this may be viable, as each edge of the network is covered exactly once in a cherry decomposition by Lemma~\ref{lem:CherryCoverCount}.

On the algorithmic front, one may find a cherry cover for a tree-based network and an acyclic cherry cover for an orchard network in polynomial time.
For orchard networks, we may find reducible pairs, cover the edges involved using the steps outlined in the proof of Theorem~\ref{thm:CherryDecompositionorchard}, reduce the shape, and continue until an acyclic cherry cover is obtained.
Since we may pick reducible pairs from orchard networks in any order \cite{janssen2018solving, erdHos2019class}, this bottom-up approach provides a polynomial time algorithm for finding an acyclic cherry cover of an orchard network.
For tree-based networks, we first find a base tree in polynomial time with the matching approach used in the proof of Theorem~$3.4$ in~\cite{jetten2018nonbinary}.
Then, one may follow the steps outlined in the proof of Theorem~\ref{thm:CherryDecompositiontree-based} of this paper to convert the cherry cover of this base tree to a cherry cover of the network in polynomial time.
Without the base tree however, it is not clear if there is a systematic way of obtaining a cherry cover.
Indeed, it is not enough to naively cover the edges in any fashion (e.g., bottom-up), as shown in Figure~\ref{fig:NaiveCovering}.
We wonder if it would be possible to directly obtain a cherry cover of a tree-based network without first having to find a base tree; and if this is the case, can it be done faster than the algorithm presented in~\cite{jetten2018nonbinary}?

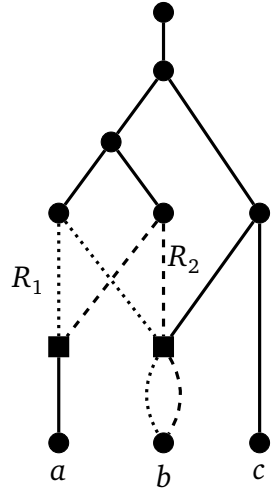
\begin{figure}
	\centering
	\begin{tikzpicture}[every node/.style={draw, circle, fill, inner sep=0pt},
	square/.style={regular polygon, regular polygon sides=4}]
    	\tikzset{edge/.style={very thick}}
	       	\node[] (root) at (0,5)	{a};
	       	
	       	\node[below = 0.5 of root] 				(t1)	{a};
		\node[below left = 0.75 and 0.5 of t1]		(t2)	{a};
		\node[below left = 0.75 and 0.5 of t2]		(t3)	{a};
		\node[below right = 0.75 and 0.5 of t2]	(t4)	{a};
		\node[right = 1 of t4]					(t5)	{a};
			
		\node[square, below = 1.5 of t3]	(r1)	{a};
		\node[square, below = 1.5 of t4]	(r2)	{a};

		\node[below = 1 of r1]		(l1)	{a};
		\node[below = 1 of r2]		(l2)	{a};
		\node[right = 1 of l2]			(l3)	{a};
		
		\draw[edge]			(root) -- (t1);
		\draw[edge]			(t1)--(t2);
		\draw[edge]			(t1)--(t5);
		\draw[edge]			(t2)--(t3);
		\draw[edge]			(t2)--(t4);
		
		\draw[edge, dotted]		(t3)--(r1) node[draw = none, fill = none, left = 1mm, midway]{\large $R_1$};
		\draw[edge, dashed]		(t4)--(r2) node[draw = none, fill = none, above right = 1mm, midway]{\large $R_2$};
		
		\draw[edge, dotted]		(t3)--(r2);
		\draw[edge, dashed]		(t4)--(r1);
		\draw[edge]			(t5)--(r2);
		
		\draw[edge]				(r1)--(l1);
		\draw[edge, bend left, dashed]	(r2)edge(l2);
		\draw[edge, bend right, dotted]	(r2)edge(l2);
		\draw[edge]				(t5)--	(l3);
		
               	\node[draw = none, fill = none, below=1mm of l1]  	(a) {\large $a$};
               	\node[draw = none, fill = none, below=1mm of l2]  	(b) {\large $b$};
               	\node[draw = none, fill = none, below=1mm of l3]	(c) {\large $c$};
			
	 \end{tikzpicture}
	 \caption{The bulged version of the tree-based network from Figure~\ref{fig:TBCherryCover}, in which we cover some of the edges with arbitrary reticulated cherry shapes~$R_1$ and~$R_2$.
	 Since the edge incident to the leaf~$a$ can no longer be covered by any reticulated cherry shape, there exists no cherry cover that contains both~$R_1$ and~$R_2$.}
	 \label{fig:NaiveCovering}

\end{figure}



\bibliographystyle{alpha}

\end{document}